\documentclass{article}

\usepackage{arxiv}

\usepackage[utf8]{inputenc} % allow utf-8 input
\usepackage[T1]{fontenc}    % use 8-bit T1 fonts
\usepackage{hyperref}       % hyperlinks
\usepackage{url}            % simple URL typesetting
\usepackage{booktabs}       % professional-quality tables
\usepackage{amsfonts}       % blackboard math symbols
\usepackage{nicefrac}       % compact symbols for 1/2, etc.
\usepackage{microtype}      % microtypography
\usepackage{lipsum}

\usepackage{enumerate}
\usepackage{amsmath}
\usepackage{amssymb}
\usepackage{graphicx}
\usepackage{float} %figures

\newtheorem{Theorem}{Theorem}

\newtheorem{Proposition}[Theorem]{Proposition}
\newtheorem{Remark}[Theorem]{Remark}

\newenvironment{proof}[1][Proof]{\noindent\textbf{#1.} }{\ \rule{0.5em}{0.5em}}

\title{Entropy monotonicity and superstable cycles for the quadratic family revisited}

\author{Jos\'{e} M. Amig\'{o} and Angel Gim\'{e}nez\\
Centro de Investigaci\'{o}n Operativa, \\  Universidad Miguel Hern\'{a}ndez de Elche, \\ Avda. de la Universidad s/n, 03202 Elche, Spain \\
  \texttt{jm.amigo@umh.es, a.gimenez@umh.es} 
}

\begin{document}
%%%%%%%%%%%%%%%%%%%%%%%%%%%%%%%%%%%%%%%%%%%%%%%

%%%%%%%%%%%%%%%%%%%%%%%%%%%%%%%%%%%%%%%%%%%%%%%
\maketitle

\begin{abstract}
The main result of this paper is a proof using real analysis of the monotonicity of the topological entropy for the family of quadratic maps, sometimes called Milnor's Monotonicity Conjecture. In contrast, the existing proofs rely in one way or another on complex analysis. Our proof is based on tools and algorithms previously developed by the authors and collaborators to compute the topological entropy of multimodal maps. Specifically, we use the number of transverse intersections of the map iterates with the so-called critical line. The approach is technically simple and geometrical. The same approach is also used to briefly revisit the superstable cycles of the quadratic maps, since both topics are closely related.
\end{abstract}

% keywords can be removed

%\textbf{AMS Subject Classification:} 37E05, 39A50, 37G35, 37N35

%%%%%%%%%%%%%%%%%%%%%%%%%%%%%%%%%%%%%%%%%%%%%%%
\section{Introduction}\label{section1}

Topological entropy is one of the main quantifiers of complexity in
continuous dynamics. First of all, it is a tight upper bound of all
measure-preserving dynamics generated by a given continuous self-map of a
compact metric space \cite{Walters1982}. Also in metric spaces, topological
entropy measures the growth rate of the number of ever longer orbits up to a
small error \cite{Bowen1971}. Its analytical calculation is only feasible in
some special cases, though. For one-dimensional dynamics, where
transformations can be supposed to be continuous and piecewise monotone
(multimodal) for practical purposes, a number of numerical algorithms based
on symbolic representations of the orbits have been developed. Examples
include kneading invariants \cite{Milnor1988}, min-max symbols \cite{Dias1},
ordinal patterns \cite{Bandt2002}, context trees \cite{Hirata2003} and more.
Precisely, this paper is the outgrowth of previous work by the authors and
collaborators on the numerical computation of the topological entropy of
multimodal maps using min-max symbols \cite{Amigo2012,Amigo2014,Amigo2015}.
At the heart of our algorithms is the number of transverse intersections
(i.e., \textquotedblleft X-crossings\textquotedblright ) of a multimodal map
and its iterates with the so-called critical lines. In this paper we also
show the potential of this concept in regard to theoretical issues. To this
end we revisit two well-traveled topics in one-dimensional dynamics:

\begin{description}
\item[(i)] the monotonicity of the topological entropy for the family of
quadratic maps, and

\item[(ii)] some basic properties of the periodic orbits of its critical
point (superstable cycles).
\end{description}

\noindent Next, we elaborate a bit on these two topics.

The family of quadratic maps (or quadratic family) is composed of the
logistic maps $f_{\mu }(x)=4\mu x(1-x)$, $0\leq \mu \leq 1$, or, for that
matter, any other dynamically equivalent maps; actually, we will use the
maps $q_{t}(x)=t-x^{2}$, $0\leq t\leq 2$, because they are algebraically
handier. When the topological entropy of multimodal maps was studied in the
1980s, the numerical results indicated that the topological entropy of the
quadratic family was a monotone function of the parameter. This property
entered the literature as \textit{Milnor's Monotonicity Conjecture},
although what he actually conjectured was the connectivity of the isentropes
(i.e., the sets of parameters for which the topological entropy is constant)
of the cubic maps in \cite{Milnor1992}, when the monotonicity of the
topological entropy for the quadratic family had already been proved by
himself (in collaboration with W. Thurston) \cite{Milnor1988} as well as by
other authors \cite{Douady1984,Douady1995,Tsujii2000}. According to \cite{Bruin2013,Bruin2015}, all these proofs use that the quadratic map can be
extended to the complex plane and require tools from complex analysis. At
variance, the proof of Milnor's Monotonicity Conjecture presented in this
paper (Section \ref{section4}, Theorem \ref{Theo9}) uses only real analysis. The conjecture was
later generalized to multimodal maps and was recently proved in \cite{Bruin2015}.

Points where a multimodal map achieves its local extrema are generically
called critical points (also when the map is not differentiable there); for
example, the critical point of $f_{\mu }$ is $x=1/2$ for all $\mu $, while
the critical point of $q_{t}$ is $x=0$ for all $t$. Orbits generated by any
of the critical points play an important role not only in symbolic dynamics
(via, e.g., the kneading invariants) but also in the stability of fixed
points and periodic attractors. Thus, unimodal maps with a negative
Schwarzian derivative (except at the critical point) and an invariant
boundary, such as the quadratic maps, have at most one stable periodic
orbit, namely, the one (if any) whose attraction basin contains the critical
point \cite{Singer1978,Melo1993}; these are the periodic attractors that can be seen
in the bifurcation diagram (Section \ref{section3}). On the other hand, if the critical
point of a quadratic map is eventually periodic, then the periodic cycle is
unstable \cite{Douady1984}.

However, we will only touch upon stability in passing. The reason why we
include the superstable cycles of the quadratic maps in this paper is
two-fold. First and least, the mathematical techniques used to deal with
both topics, entropy monotonicity and superstable cycles, are similar, so we
can easily exploit this fact. More importantly, the relationship between
these topics is deeper than might be thought. Indeed, Thurston's Rigidity, a
result on the periodic orbits of the critical points of quadratic maps,
implies Milnor's Monotonicity Conjecture for the quadratic maps and it is
necessary (in a generalized version) to prove the case of polynomial maps of
higher degrees \cite{Bruin2015}. Here we will give only a general idea of
this relationship. In addition, we will briefly discuss the
\textquotedblleft dark lines\textquotedblright\ through the chaotic bands of
the bifurcation diagram of the quadratic family, which relate to the orbit
of the critical point, as well as the parameter values for which the
critical point is eventually periodic (Misiurewicz points).

This paper is organized as follows. In Section \ref{section2} we introduce the
mathematical background needed for the following sections. In particular, we
introduce the expression of the topological entropy for multimodal maps via
the number of transverse crossings of its iterates with the critical lines.
More specific concepts and tools that refer to the family of the quadratics
maps (Section \ref{section3}) are discussed in Section \ref{section31} (root branches) and Section \ref{section32} (smoothness domains of the root branches). Root branches have many
interesting properties but we only address those we need for our purposes.
The materials of Sections \ref{section2} and \ref{section3} will then be used in two complementary
ways. The proof in Section \ref{section4} that the smoothness domains of the root
branches are half-intervals lead to the monotonicity of the topological
entropy for the quadratic family. The bifurcation points of some root
branches lead to the basic properties and parameter values of the
superstable cycles of the quadratic family (Section \ref{section51}). The latter topic
will be completed with a short digression on the eventually periodic orbits
of the critical point (Section \ref{section52}).

\section{Mathematical preliminaries}\label{section2}

\subsection{Multimodal maps}\label{section21}

Let $I$ be a compact interval $[a,b]\subset \mathbb{R}$ and $f:I\rightarrow
I $ be a \textit{piecewise monotone} continuous map. Such a map is called $l$-\textit{modal} if $f$ has local extrema at precisely $l$ interior points $c_{1}<...<c_{l}$. Moreover, we assume that $f$ is strictly monotone in each of the $l+1$
intervals\begin{equation*}
I_{1}=[a,c_{1}),I_{2}=(c_{1},c_{2}),...,I_{l}=(c_{l-1},c_{l}),I_{l+1}=(c_{l},b]\text{.}
\end{equation*}The points $c_{1},...,c_{l}$ are called \textit{critical }or \textit{turning}
\textit{points} and their images $f(c_{1})$, ..., $f(c_{l})$ are the \textit{critical values} of $f$. These maps are also referred to as multimodal maps
(for a general $l$) and unimodal maps (if $l=1$). We denote the set of $l$-modal maps by $\mathcal{M}_{l}(I)$, or just $\mathcal{M}_{l}$ if the
interval $I$ is clear from the context or unimportant for the argument. $f\in \mathcal{M}_{l}(I)$ is said to have \textit{positive} (resp. \textit{negative}) \textit{shape} if $f(c_{1}$) is a maximum (resp. minimum); here
and hereafter, all extrema are meant to be local unless stated otherwise.
Thus, if $f$ has positive shape, then $f$ is strictly increasing in the
intervals with odd subindex ($I_{\text{odd}}$) and strictly decreasing in
the intervals with even subindices ($I_{\text{even}}$).

For $n\geq 0$, $f^{n}$ denotes the $n$th iterate of $f$, where $f^{0}$ is
the identity map$.$ Since $f$ is continuous and piecewise strictly monotone,
so is $f^{n}$ for all $n\geq 1$. The proof of the following Proposition is direct
(see \cite[Lemma 2.2]{Amigo2014}).

\begin{Proposition}
\label{Propo1} Let $f\in \mathcal{M}_{l}(I)$ with positive shape and $n\geq
1 $. We have:\begin{equation}
f^{n+1}(x)\text{ is a maximum if }\left\{ 
\begin{array}{l}
\text{(i) }f^{n}(x)=c_{\text{odd}}\text{,} \\ 
\text{(ii) }f^{n}(x)\in I_{\text{even}}\text{ and }f^{n}(x)\text{ is a
minimum, or} \\ 
\text{(iii) }f^{n}(x)\in I_{\text{odd}}\text{ and }f^{n}(x)\text{ is a
maximum,}\end{array}\right.  \label{Lemma1A}
\end{equation}and\begin{equation}
f^{n+1}(x)\text{ is a minimum if }\left\{ 
\begin{array}{l}
\text{(i) }f^{n}(x)=c_{\text{even}}\text{,} \\ 
\text{(ii) }f^{n}(x)\in I_{\text{odd}}\text{ and }f^{n}(x)\text{ is a
minimum, or} \\ 
\text{(iii) }f^{n}(x)\in I_{\text{even}}\text{ and }f^{n}(x)\text{ is a
maximum.}\end{array}\right.  \label{Lemma1B}
\end{equation}
\end{Proposition}

If $f$ has negative shape, then replace \textquotedblleft$f^{n+1}(x)$ is a
maximum if\textquotedblright\ by \textquotedblleft$f^{n+1}(x)$ is a minimum
if\textquotedblright\ in (\ref{Lemma1A}), and the other way around in (\ref{Lemma1B}).

Apply Proposition \ref{Propo1} to $f^{n}$, $f^{n-1}$, ..., $f$ \ to conclude
that $f^{n+1}$ has local extrema at all $x\in I$ such that $f^{k}(x)=c_{i}$
for $k=0,1,..,n$ and some $i$. This proves:

\begin{Proposition}
\label{Propo1B} Let $f\in \mathcal{M}_{l}(I)$ and $n\geq 1$. Then $f^{n}$
has local extrema at the critical points and their preimages up to order $n-1 $.
\end{Proposition}

For $n\geq 1$, let\begin{equation}
s_{n,i}=\#\{x\in (a,b):f^{n}(x)=c_{i},\,\,\,f^{k}(x)\neq c_{j}\text{ for }0\leq k\leq n-1\text{, }1\leq j\leq l\},  \label{sni}
\end{equation}i.e., the number of \textit{interior} \textit{simple} zeros of the function $f^{n}(x)-c_{i}$, and set 
\begin{equation}
s_{n}=\sum_{i=1}^{l}s_{n,i}  \label{snu}
\end{equation}for the total number of such zeros. For the convenience of notation,
definition (\ref{sni}) can be extended to $n=0$: $s_{0,i}=\#\{x\in
(a,b):x=c_{i}\}=1$, so that $s_{0}=l$.

In the case of differentiable maps (to be considered in Sections \ref{section3}--\ref{section5}), $s_{n,i}$ amounts geometrically to the number of \textit{transverse}
intersections of $y=f^{n}(x)$ with the $i$th critical line $y=c_{i}$.
Indeed, by the chain rule of derivation,\begin{equation}
\frac{df^{n}}{dx}(x)=\prod\limits_{k=0}^{n-1}\frac{df}{dx}(f^{k}(x)).
\label{derivf^n}
\end{equation}Therefore, if $f^{k}(x)\neq c_{j}$ for all $0\leq k\leq n-1$ and $1\leq
j\leq l$, then $df^{n}(c_{i})/dx\neq 0$. A
solution $x^{\ast }$ of $f^{n}(x)-c_{i}=0$ such that $df^{n}(x^{\ast })/dx=0$
corresponds to a tangential intersection of the curve $y=f^{n}(x)$ with the
critical line $y=c_{i}$. Abusing the language, we will speak of transverse
and non-transverse intersections in the general case too. Incidentally, Equation
(\ref{derivf^n}) proves Proposition \ref{Propo1B} for differentiable maps.

Next, let $e_{n}$ be the \textit{number of local extrema }of\textit{\ }$f^{n} $.

\begin{Proposition}
\label{Propo2} Let $f\in \mathcal{M}_{l}(I)$ and $n\geq 0$. Then,\begin{equation}
e_{n+1}=e_{n}+s_{n}.  \label{1.0}
\end{equation}
\end{Proposition}

\begin{proof}
If $n=0$, then $e_{0}=0$ and $s_{0}=l$, so that $e_{0}+s_{0}$ gives the
right answer $e_{1}=l$.

Suppose now that $n\geq 1$ and $f^{n+1}$ has a local extremum at $x_{0}\in I$, so that $e_{n+1}$ is the number of such $x_{0}$'s. According to
Proposition \ref{Propo1}, there are two exclusive possibilities:

(a) $f^{n}(x_{0})=c_{i}$ for some $1\leq i\leq l$ (Proposition \ref{Propo1}(i)); or

(b) $f^{n}(x_{0})\neq c_{i}$ for all $1\leq i\leq l$ and $f^{n}$ has a local
extremum at $x_{0}$ (Proposition \ref{Propo1}(ii) and (iii)).

In turn, (a) subdivides according to whether $x_{0}$ is a transverse or a
tangential intersection of $y=f^{n}(x)$ with the critical line $y=c_{i}$:

(a1) $f^{n}(x_{0})=c_{i}$ and $f^{k}(x_{0})\neq c_{j,}$ for all $0\leq k\leq
n-1$, $1\leq j\leq l$.

(a2) $f^{n}(x_{0})=c_{i}$ and $f^{k}(x_{0})=c_{j}$ for \ some $k$ and $j$, $0\leq k\leq n-1$, $1\leq j\leq l$.

Therefore, each $x_{0}\in I$ that contributes to $e_{n+1}$ contributes to $s_{n}$ (if case (a1) holds) or, otherwise, to $e_{n}$ (if case (a2) or (b)
holds). The bottom line is Equation (\ref{1.0}).
\end{proof}

Figure \ref{figure1} illustrates Equation (\ref{1.0}) for the bimodal map\begin{equation}
f(x)=9.375x^{3}-15.4688x^{2}+6.75x+0.1,  \label{bimodal}
\end{equation}$I=[0,1]$, whose critical points are $c_{1}=0.3$ ($f(c_{1})=0.985938$) and $c_{2}=0.8$ ($f(c_{2})=0.4$).

In the next two sections we discuss how the transverse and tangential
intersections of $f^{n}$ with the critical lines are related to two salient
aspects of the dynamics generated by $f$: topological entropy and
superstable periodic orbits.

\begin{figure}[]
\centering
\includegraphics[height=3cm]{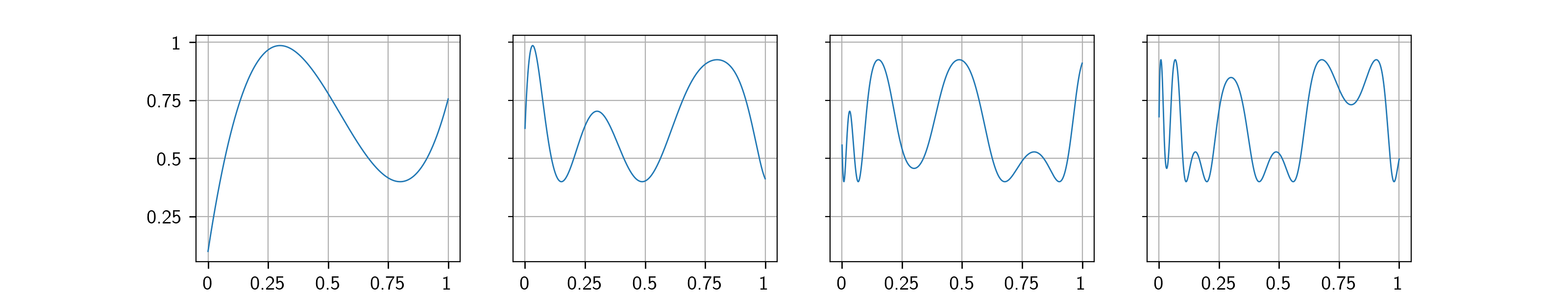}
\caption{Graphs of $f$, $f^{2}$, $f^{3}$, and $f^{4}$ for the bimodal map (\protect\ref{bimodal}). Reproduced from \protect\cite{Amigo2012}.}
\label{figure1}
\end{figure}

\subsection{Topological entropy}\label{section22}

The connection of the recursive formula (\ref{1.0}) with the topological
entropy of $f\in \mathcal{M}_{l}(I)$, $h(f)$, is readily established through
the \textit{lap number} $\ell _{n}$ of $f^{n}$, which is defined as the
number of maximal monotonicity intervals of $f^{n}$. First, replace $\ell
_{n}=e_{n}+1$ in (\ref{1.0}) to obtain 
\begin{equation}
\ell _{n+1}=\ell _{n}+s_{n}.  \label{account5}
\end{equation}The initial values $\ell _{0}=1$ and $s_{0}=l$ yield $\ell _{1}=l+1$, as it
should.

Second, use the relation \cite{Misiu} 
\begin{equation}
h(f)=\lim_{n\rightarrow \infty }\frac{1}{n}\log \ell _{n}.  \label{1.1}
\end{equation}Since $\ell _{n}\leq (l+1)^{n}$ (see e.g. \cite{Amigo2015}),\begin{equation}
h(f)\leq \log (l+1).  \label{maxh(f)}
\end{equation}

Finally, Equations (\ref{account5}) and (\ref{1.1}) lead then to the
expression\begin{equation}
h(f)=\lim_{n\rightarrow \infty }\frac{1}{n}\log \left(
1+\sum\limits_{k=0}^{n-1}s_{k}\right) ,  \label{account6}
\end{equation}which was first derived in \cite{Amigo2014}. For the general concept of
entropy, see \cite{Katok2007,Amigo2015B,Amigo2018}.

As a technical remark, the topological entropy of a continuous map (in
particular, a multimodal map $f:I\rightarrow I$) only depends on its
non-wandering set \cite{Alseda}. A point $x\in I$ is said to be \textit{non-wandering} for $f$ if for any neighborhood $U$ of $x$, there is an
integer $n\geq 1$ such that $f^{n}(U)\cap U\neq \emptyset $; otherwise, $x$
is said to be a wandering point for $f$. The \textit{non-wandering set} for $f$ consists of all the points that are non-wandering for $f$.

Equations (\ref{1.1}) and (\ref{account6}) add to other similar expressions
of $h(f)$ in terms of $e_{n}=\ell _{n}-1$, the number of $n$-periodic
points, the variation of $f^{n}$ \cite[Theorem 1.1]{Bruin2013}, etc. In this
regard, the quantities $s_{k}$ in Equation (\ref{account6}) can be viewed in
the following three different ways:

(\textbf{1}) Algebraically, $s_{k}$ is by definition (\ref{sni})-(\ref{snu})
the number of interior simple zeros of the equations $f^{k}(x)-c_{i}=0$, $i=1,2,...,l$.

(\textbf{2}) Geometrically, $s_{k}$ is the total number of transverse
intersections of the iterated map $f^{k}$ with the critical lines.

(\textbf{3}) Dynamically, $s_{k}$ is the total number of preimages of the
critical points of minimal order $k$.

Whatever the interpretation, we are going to show that $s_{k}$ is a useful
tool to study multimodal maps.

Several numerical algorithms for the topological entropy of multimodal maps
based on Equation (\ref{account6}) can be found in \cite{Amigo2012,Amigo2014,Amigo2015}, the algorithm in \cite{Amigo2015} being a
variant of the algorithm in \cite{Amigo2014} and this, in turn, a
simplification of the algorithm in \cite{Amigo2012}. The performance of the
algorithm \cite{Amigo2014} has recently been benchmarked in 
\cite{Cockram2020} with favorable results. The computation of $s_{n}$ from the values of $s_{0},...,s_{n-1}$ is possible via the so-called
min-max sequences \cite{Dias1}, which are closely related to the kneading
sequences \cite{Milnor1988,Melo1993}. As compared to the kneading symbols,
the min-max symbols contain additional information on the minimum/maximum
character of the critical values $f^{n}(c_{i})$, $1\leq i\leq l$, with
virtually no extra computational penalty \cite{Amigo2012,Amigo2014}. The
geometrical properties of the min-max symbols were studied in \cite{Dilao2012} and \cite{Amigo2012} for twice-differentiable uni- and
multimodal maps, respectively, and in \cite{Amigo2014,Amigo2015} for just
continuous multimodal maps. A brief overview is given in the Introduction of 
\cite{Amigo2015}.

Let $f_{t}\in \mathcal{M}_{l}(I_{t})$ be a one-parametric family of $l$-modal maps whose parameter $t$ ranges in an interval $J\subset \mathbb{R}$.
Denote by $s_{n}(t)$ the total number of transverse intersections of $y=f_{t}^{n}(t)$ with the critical lines. From (\ref{account6}) and the
monotonicity of the logarithmic function it follows:

\begin{Proposition}
\label{Propo3} Let $f_{t}\in \mathcal{M}_{l}(I_{t})$, and $t_{1},t_{2}\in J$
with $t_{1}<t_{2}$. Suppose $s_{n}(t_{1})\leq s_{n}(t_{2})$ for all $n\geq
n_{0}$. Then $h(f_{t_{1}})\leq h(f_{t_{2}})$.
\end{Proposition}

As we will see in Section \ref{section4}, Proposition \ref{Propo3} provides a handle to
prove the monotonicity of the topological entropy for the family of
quadratic maps. We mentioned already in the Introduction that, according to 
\cite{Bruin2013,Bruin2015}, the existing monotonicity proofs \cite{Douady1984,Milnor1988,Douady1995,Tsujii2000}\textbf{\ }rely in one way or
another on complex analysis. Unlike them, our approach uses real analysis.
Let us remind at this point that the topological entropy of a family of
unimodal maps labeled by some natural parameter (such as its critical value)
is not usually monotone, even under very favorable assumptions \cite{Bruin1995}. More generally, let $f_{v}$ be a polynomial map parametrized by
its critical \textit{values} $v=(v_{1},...,v_{l})$. Then, according to \cite[Theorem 1.1]{Bruin2013}, for $l\geq 2$ there exist fixed values of $v_{2},...,v_{l}$ such that the map $v_{1}\mapsto h(f_{v})$ is not monotone.
For multimodal maps, monotonicity of the map is replaced by the connectivity
of the isentropes \cite[Theorem 1.2]{Bruin2015}. See also \cite{Bruin2015}
for related results and open conjectures.

\subsection{Superstable periodic orbits}\label{section23}

Let $x_{0}\in I$ and set $x_{k}=f^{k}(x_{0})=f(x_{k-1})$ for $k\geq 1$.
Suppose for the time being that $f$ is differentiable and a critical point $c_{i}$ is periodic with prime period $p$. Then, each point of the orbit $\mathcal{O}(c_{i})=\{c_{i}\equiv x_{0},x_{1},...,x_{p-1}\}$ is a fixed point
of $f^{p}$: $f^{p}(x_{j})=x_{j+p}=x_{j}$ for $0\leq j\leq p-1$. $\mathcal{O}(c_{i})$ is said to be \textit{superstable} because (see Equation (\ref{derivf^n})) 
\begin{equation}
\frac{df^{p}}{dx}(x_{j})=\frac{df}{dx}(x_{0})\frac{df}{dx}(x_{1})...\frac{df}{dx}(x_{p-1})=0\;\;\text{for\ }j=0,1,...,p-1,  \label{super}
\end{equation}since $df(x_{0})/dx\equiv df(c_{i})/dx=0$. In other words, $df^{p}(x_{j})/dx$
(whose absolute value quantifies the stability of the fixed points $c_{i}$, $x_{1}$,..., $x_{p-1}$ of $f^{p}$ ) vanishes at each point of the periodic
orbit.

On the other hand, $f^{n}$ has local extrema at all critical points for $n\geq 1$, so that the periodicity condition $f^{p}(c_{i})=c_{i}$ amounts to
a tangential intersection of the curve $y=f^{p}(x)$ and the critical line $y=c_{i}$ at $x=c_{i}$. Therefore, while the transverse intersections of $f^{n}$ with the critical lines are the only input needed to calculate the
topological entropy of multimodal maps, the tangential intersections, if
any, are the main ingredient of the periodic orbits (\textit{cycles}) of a
critical point. All in all, the intersections of $f^{n}$ with the critical
lines, whether transverse or tangential, give information about the
dynamical complexity and superstability of the orbits.

\section{Application case: Quadratic maps}\label{section3}

Quadratic maps have been the workhorse of chaotic dynamics for two good
reasons: their dynamic exhibits a mind-boggling complexity despite being
algebraically so simple and, precisely because of this simplicity, many of
their dynamical properties are amenable to analytical scrutiny. We consider
henceforth the family of the real quadratic maps\begin{equation}
q_{t}(x)=t-x^{2},  \label{0.0}
\end{equation}where $x\in \mathbb{R}$ and $0\leq t\leq 2$. The critical point and the
critical value of $q_{t}$ are $c=0$ and\ $q_{t}(0)=t$, respectively, so the
critical line $y=0$ is the $x$-axis in the Cartesian plane $\{(x,y)\in 
\mathbb{R}^{2}\}$. The quadratic family has two fixed points,\begin{equation}
x_{fix,1}(t)=-\frac{1}{2}\left( 1+\sqrt{1+4t}\right) \leq
-1,\;\;x_{fix,2}(t)=\frac{1}{2}\left( -1+\sqrt{1+4t}\right) \geq 0.
\label{fix}
\end{equation}Therefore, an invariant finite interval $I_{t}$, i.e., $q_{t}(I_{t})\subset
I_{t}$, where defining a dynamic generated by $q_{t}$, is\begin{equation}
I_{t}=[x_{fix,1}(t),-x_{fix,1}(t)]=\left[ -\tfrac{1}{2}(1+\sqrt{1+4t}),\tfrac{1}{2}(1+\sqrt{1+4t})\right] .  \label{I_t}
\end{equation}It holds $I_{0}=[-1,1]\subset I_{t}\subset \lbrack -2,2]=I_{2}$. Moreover, 
\begin{equation}
q_{t}(-x_{fix,1}(t))=q_{t}(x_{fix,1}(t))=x_{fix,1}(t),  \label{boundary}
\end{equation}so that the boundary of $I_{t}$, $\partial
I_{t}=\{x_{fix,1}(t),-x_{fix,1}(t)\}$, is also invariant: $q_{t}(\partial
I_{t})=\{x_{fix,1}(t)\}\subset \partial I_{t}$. Since all $x\notin I_{t}$
escape to $-\infty $ under iterations of $q_{t}$, the set $I_{t}$ contains
the non-wandering set of $q_{t}$.

See Figure \ref{figure2} for some instances of the quadratic family. The
bifurcation diagram of $q_{t}(x)$ in Figure \ref{figure3} shows that the
asymptotic dynamics of the quadratic family (chaotic attractors, along with
stable fixed points and periodic orbits) lives in the interval $-2\leq x\leq
2${.} After the period-doubling cascade, {chaos onset occurs at the \textit{Feigenbaum point} $t_{F}=1.401155...$, i.e., the topological entropy of }$q_{t}$ is positive for $t>t_{F}$.

\begin{figure}[tbp]
\centering
\includegraphics[height=3cm]{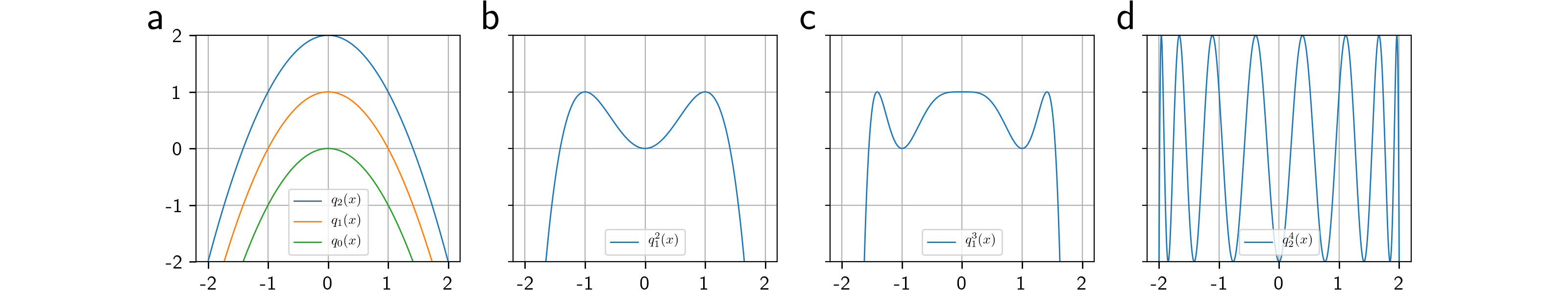}
\caption{(a) Bottom to top: $q_{0}(x)$, $q_{1}(x)$, $q_{2}(x)$. (b) Graph of 
$q_{1}^{2}(x)$. (c) Graph of $q_{1}^{3}(x)$. (d) Graph of $q_{2}^{4}(x)$.}
\label{figure2}
\end{figure}

\begin{figure}[tbp]
\centering
\includegraphics[height=4cm]{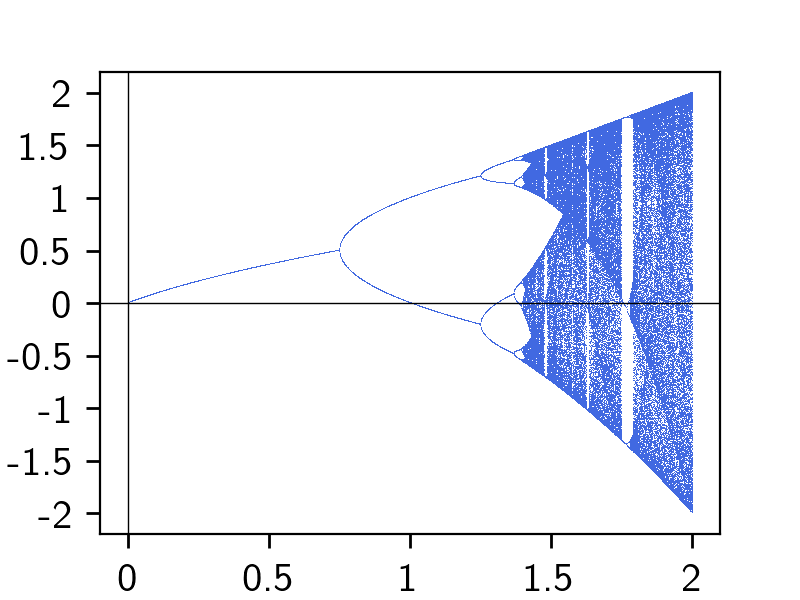}
\caption{Bifurcation diagram of $q_{t}(x)$, $0\leq t\leq 2$.}
\label{figure3}
\end{figure}

The dynamical systems generated by $q_{t}(x)$, where $x\in I_{t}$ and $0\leq
t\leq 2$, and the more popular logistic maps $f_{\mu }(z)=4\mu z(1-z)$,
where $0\leq z\leq 1$ and $\frac{1}{2}\leq \mu \leq 1$, are conjugate to
each other via the affine transformation $\varphi :[0,1]\rightarrow \lbrack
-2\mu ,2\mu ]$ defined as\begin{equation}
x=\varphi (z)=4\mu z-2\mu \text{\ \ and }t=2\mu (2\mu -1)  \label{phi(z)}
\end{equation}or 
\begin{equation}
z=\varphi ^{-1}(x)=\frac{x}{4\mu }+\frac{1}{2}\text{\ \ with }\mu =\frac{1}{4}(1+\sqrt{1+4t}).  \label{invphi(x)}
\end{equation}Thus, $\left. q_{0}\right\vert _{[-1,1]}$ is conjugate to $\left.
f_{0.5}\right\vert _{[0,1]}$, and $\left. q_{2}\right\vert _{[-2,2]}$ to $\left. f_{1}\right\vert _{[0,1]}$. Note that $-2\mu =x_{fix,1}$, so $I_{t}=[-2\mu ,2\mu ]$.

An advantage of the quadratic map (\ref{0.0}) is that the transverse (resp.
tangential) intersections of $y=q_{t}^{n}(x)$ with the critical line
correspond to the simple (resp. multiple) roots of $q_{t}^{n}(x)$, a
polynomial of degree $2^{n}$. Since $q_{t}(x)$ is unimodal ($l=1$), Equation
(\ref{snu}) simplifies to 
\begin{equation}
s_{n}(t)=\#\{x\in \mathring{I}_{t}:q_{t}^{n}(x)=0,\,\,\,q_{t}^{k}(x)\neq 0\text{ for }0\leq k\leq n-1\},  \label{s^n(t)}
\end{equation}where $\mathring{I}_{t}=I_{t}\backslash \partial I_{t}$ is the interior of $I_{t}$. Therefore, $s_{n}(t)$ stands for the number of \textit{simple} zeros
of $q_{t}^{n}(x)$ in $\mathring{I}_{t}$ or, equivalently, for the number of 
\textit{transverse} intersections of the curve $y=q_{t}^{n}(x)$ with the
critical line $y=0$. We show in Remark \ref{Remark2} below that $\mathring{I}_{t}$ contains all zeros of $q_{t}^{n}(x)$, therefore $\mathring{I}_{t}$ can
be safely replaced by $\mathring{I}_{2}=(-2,2)$ (or $\mathbb{R}$, for that
matter) in Equation (\ref{s^n(t)}).

\subsection{Root branches}\label{section31}

We set out to study the real solutions of the equation $q_{t}^{n}(x)\equiv
t-(q_{t}^{n-1}(x))^{2}=0$, $n\geq 1$, which is a polynomial equation of
degree $2^{n}$ in $x$. If $\bar{x}$ is a solution, then $-\bar{x}$ is also a
solution since $q_{t}^{n}(-x)=q_{t}^{n}(x)$.

The following two cases are trivial:\ (i) for $t=0$, $q_{0}^{n}(x)\equiv
-x^{2^{n}}=0$ has the $2^{n}$-fold solution $x=0$; (ii) for $t=2$, $q_{2}^{n}(x)=0$ has $2^{n}$ simple solutions in $(-2,2)$, namely,\begin{equation}
\bar{x}_{\sigma _{1},...,\sigma _{n}}=\sigma _{1}\sqrt{2+\sigma _{2}\sqrt{2+...+\sigma _{n}\sqrt{2}}},  \label{x^bar}
\end{equation}where $\sigma _{1},...,\sigma _{n}\in \{+,-\}$. Alternatively, the roots $\bar{x}_{\sigma _{1},...,\sigma _{n}}$ have the following trigonometric
closed-form \cite[Problem 183]{Polya1970}:\begin{equation}
\bar{x}_{\sigma _{1},...,\sigma _{n}}=2\sin \left( \frac{\pi }{4}\sum_{k=1}^{n}\frac{\sigma _{1}\sigma _{2}\cdots \sigma _{k}}{2^{k-1}}\right) .  \label{x_sigma(2)}
\end{equation}

In the general case, consider the map $F_{n}:\mathbb{R}\times \lbrack
0,2]\rightarrow \mathbb{R}$ defined as $F_{n}(x,t)=$ $q_{t}^{n}(x)$ and the
point $(\bar{x},\bar{t})=(\bar{x}_{\sigma _{1},...,\sigma _{n}},2)$, so that 
$F_{n}(\bar{x},\bar{t})=$ $q_{2}^{n}(\bar{x}_{\sigma _{1},...,\sigma
_{n}})=0 $ and\begin{equation}
\frac{\partial F_{n}}{\partial x}(\bar{x},\bar{t})=\frac{dq_{\bar{t}}^{n}}{dx}(\bar{x})=\prod\limits_{k=0}^{n-1}\frac{dq_{2}}{dx}(q_{2}^{k}(\bar{x}_{\sigma _{1},...,\sigma _{n}}))\neq 0  \label{IFThm}
\end{equation}since $q_{2}^{0}(\bar{x}_{\sigma _{1},...,\sigma _{n}})=\bar{x}_{\sigma
_{1},...,\sigma _{n}}\neq 0$ and $q_{2}^{k}(\bar{x}_{\sigma _{1},...,\sigma
_{n}})=-\bar{x}_{\sigma _{k+1},...,\sigma _{n}}\neq 0$ for $k=1,...,n-1$. By
the \textit{Implicit Function Theorem}, there exists a neighborhood $U\subset \lbrack 0,2]$ of $\bar{t}=2$ and a unique smooth function $\phi
_{\sigma _{1},...,\sigma _{n}}:U\rightarrow \mathbb{R}$ such that $\phi
_{\sigma _{1},...,\sigma _{n}}(2)=\bar{x}_{\sigma _{1},...,\sigma _{n}}$ and 
$q_{t}^{n}(\phi _{\sigma _{1},...,\sigma _{n}}(t))\equiv t-q_{t}^{n-1}(\phi
_{\sigma _{2},...,\sigma _{n}}(t))^{2}=0$, i.e., 
\begin{equation}
\phi _{\sigma _{1},...,\sigma _{n}}(t)=\sigma _{1}\sqrt{t+\phi _{\sigma
_{2},...,\sigma _{n}}(t)}=...=\sigma _{1}\sqrt{t+\sigma _{2}\sqrt{t+...+\sigma _{n}\sqrt{t}}},  \label{phi=x}
\end{equation}for all $t\in U$. Therefore, in this case the \textquotedblleft
implicit\textquotedblright\ functions $\phi _{\sigma _{1},...,\sigma
_{n}}(t) $ are explicitly known, and\begin{equation}
-2<\phi _{\sigma _{1},...,\sigma _{n}}(t)<2  \label{phi<2}
\end{equation}for all $n\geq 1$, $\sigma _{1},...,\sigma _{n}\in \{+,-\}$, and $0\leq
t\leq 2$.

The functions $\phi _{\sigma _{1},...,\sigma _{n}}(t)$ will be generically
called \textit{root branches} of $q_{t}^{n}(x)$; notice that the sign of $\phi _{\sigma _{1},...,\sigma _{n}}(t)$ depends on $\sigma _{1}$, hence $\phi _{-\sigma _{1},\sigma _{2},...,\sigma _{n}}(t)=-\phi _{\sigma
_{1},\sigma _{2},...,\sigma _{n}}(t)$. When the components are not
important, we shorten the notation and write $(\sigma _{1},...,\sigma
_{n})=\sigma $. We call $n$, the number of components of $\sigma $, the 
\textit{rank} of the \textit{signature} $\sigma $ and denote it by $\left\vert \sigma \right\vert $. Likewise, we call $\left\vert \sigma
\right\vert $ the rank of $\phi _{\sigma }(t)$, so $q_{t}^{\left\vert \sigma
\right\vert }(\phi _{\sigma }(t))=0$. Sometimes we write $\pm $ in a
component of a signature to refer to both branches. If $\sigma $ is a final
segment of the signature $\rho $, we say that $\phi _{\rho }(t)$ is a 
\textit{successor} of $\phi _{\sigma }(t)$; likewise if $\sigma $ is an
initial segment of the signature $\rho $, we say the $\phi _{\rho }(t)$ is a 
\textit{predecessor} of $\phi _{\sigma }(t)$.

Let us pause at this point to address a few basic properties of the root
branches. We denote by $\mathbf{dom\,}\phi _{\sigma }$ the \textit{definition domain} of $\phi _{\sigma }(t)$, that is, the points in the
parametric interval $[0,2]$ where the right hand side of Equation (\ref{phi=x}) exists. In view of (\ref{phi=x}), the definition domains of $\phi
_{\pm ,\sigma }(t)$, the two immediate successors\ of $\phi _{\sigma }(t)$,
are given by 
\begin{equation}
\mathbf{dom\,}\phi _{+,\sigma }=\mathbf{dom\,}\phi _{-,\sigma }=\{0\leq
t\leq 2:t+\phi _{\sigma }(t)\geq 0\}.  \label{root8}
\end{equation}Since $\phi _{\sigma }(0)=0$ for all signatures $\sigma $ and $\phi _{\sigma
}(2)=\bar{x}_{\sigma _{1},...,\sigma _{n}}\in (-2,2)$, it holds $\{0,2\}\subset \mathbf{dom\,}\phi _{\pm ,\sigma }$ for all root branches. It
is also obvious that\begin{equation}
\mathbf{dom\,}\phi _{+,\sigma }=\mathbf{dom\,}\phi _{-,\sigma }\subset 
\mathbf{dom\,}\phi _{\sigma },  \label{root9}
\end{equation}so that the consecutive successors of $\phi _{\sigma }(t)$ have, in general,
ever smaller definition domains. The only exceptions are 
\begin{equation}
\mathbf{dom\,}\phi _{+,+,...,+}=\mathbf{dom\,}\phi _{-,+,...,+}=[0,2].
\label{full dom}
\end{equation}Examples of definition domains are the following:
\begin{align}
\mathbf{dom\,}\phi _{+,-}& =\{0\}\cup \lbrack 1,2],  \label{Exam_dom} \\
\mathbf{dom\,}\phi _{+,-,+}& =\{0\}\cup \lbrack 1.7549...,2],  \notag \\
\mathbf{dom\,}\phi _{+,-,+,-}& =\{0\}\cup \{1\}\cup \lbrack 1.3107...,2]. 
\notag
\end{align}

Figure \ref{figure4} shows the graphs of the root branches of ranks 1 to 5.
In panel (a), the 2-fold zero $\phi _{\pm }(0)=0$ correspond to the
tangential intersection of $q_{0}(x)$ with the $x$-axis in Figure \ref{figure2}(a), while the 2-fold zero $\phi _{\pm -}(1)=0$ and the two simple
zeros $\phi _{++}(1)=\sqrt{2}$ and $\phi _{-+}(1)=-\sqrt{2}$ correspond to
the tangential intersection and the two transverse intersections,
respectively, of $q_{1}^{2}(x)$ with the $x$-axis in Figure \ref{figure2}(b). In panel (b), the two simple roots $\phi _{+++}(1)=1.5538$ and $\phi
_{-++}(1)=-1.5538$ correspond to the transverse intersections of $q_{1}^{3}(x)$ with the $x$-axis in Figure \ref{figure2}(c), while the two
2-fold roots $\phi _{+\pm -}(1)=1$ and $\phi _{-\pm -}(1)=-1$ correspond to
the two tangential intersections of $q_{1}^{3}(x)$. The 16 roots $\phi
_{\sigma _{1},\sigma _{2},\sigma _{3},\sigma _{4}}(2)$ in panel (c)
correspond to the 16 transverse intersections of $q_{2}^{4}(x)$ with the $x$-axis in Figure \ref{figure2}(d). Finally, panel (d) shows together the 62
root branches of ranks 1 to 5.

\begin{figure}[tbp]
\centering
\includegraphics[height=9cm]{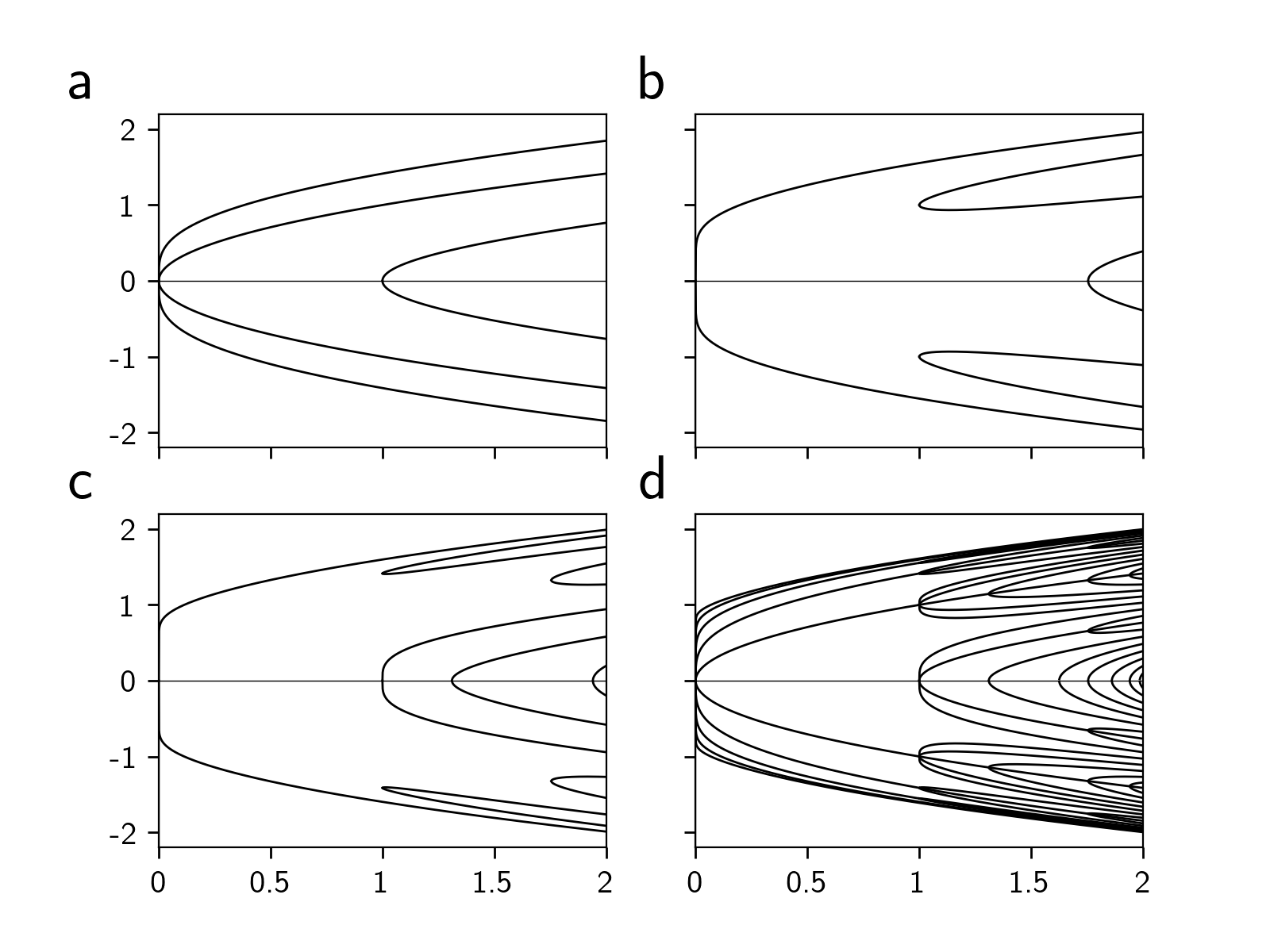}
\caption{Root branches of ranks 1 to 5; since $\protect\phi _{-,\protect\sigma }(t)=-\protect\phi _{+,\protect\sigma }(t)$, only positive branches
(upper half-plane) are specified next (top to bottom). (a) Ranks 1 and 2: $\protect\phi _{++}$, $\protect\phi _{+}$, and $\protect\phi _{+-}$. (b) Rank
3: $\protect\phi _{+++}$, $\protect\phi _{+\pm -}$, and $\protect\phi _{+-+}$. (c) Rank 4: $\protect\phi _{++++}$, $\protect\phi _{++\pm -}$, $\protect\phi _{+\pm -+}$, $\protect\phi _{+---}$, $\protect\phi _{+-+-}$, and $\protect\phi _{+-++}$. (d) Joint plot of the root branches of ranks 1 to 5;
positive branches of rank 5: $\protect\phi _{+++++}$, $\protect\phi _{+++\pm
-}$, $\protect\phi _{++\pm -+}$, $\protect\phi _{++---}$, $\protect\phi _{++-+-}$, $\protect\phi _{++-++}$, $\protect\phi _{+--++}$, $\protect\phi _{+--+-}$, $\protect\phi _{+----}$, $\protect\phi _{+-\mp -+}$, $\protect\phi _{+-+--}$, $\protect\phi _{+-++-}$, and $\protect\phi _{+-+++}$.}
\label{figure4}
\end{figure}

In the panels of Figure \ref{figure4} we see that the $2^{\left\vert \sigma
\right\vert }$ root branches $\phi _{\sigma }(t)$, $1\leq \left\vert \sigma
\right\vert \leq 5$, build $2^{\left\vert \sigma \right\vert -1}$
parabola-like curves, which we denote $\phi _{\sigma _{1},...,\sigma
_{i-1},\pm ,\sigma _{i+1},...,\sigma _{n}}(t)$ ($1\leq i\leq n$), this
notation meaning that the curves $\phi _{\sigma _{1},...,\sigma
_{i}=+,...,\sigma _{n}}(t)$ and $\phi _{\sigma _{1},...,\sigma
_{i}=-,...,\sigma _{n}}(t)$ (the branches of the parabola) emerge from a
common vertex $(t_{b},\phi _{\sigma _{1},...,\sigma _{i}=+,...,\sigma
_{n}}(t_{b}))=(t_{b},\phi _{\sigma _{1},...,\sigma _{i}=-,...,\sigma
_{n}}(t_{b}))$ with a vertical tangent. The vertex and the abscissa $t_{b}$
will be called indistinctly branching point (geometrical terminology) or
bifurcation point (dynamical terminology) of the parabola or any of its
branches. Root parabolas with the vertex on the $t$-axis, $\phi _{\pm
,\sigma _{2},...,\sigma _{n}}(t)$, are sometimes called \textit{on-line
parabolas}, otherwise \textit{off-axis parabolas}. The branching point $t_{b} $ has also a direct geometrical interpretation in state space: the
curve $q_{t_{b}}^{\left\vert \sigma \right\vert }(x)$ intersects
tangentially the $x $-axis (the critical line) at the point $x=\phi _{\sigma
_{1},...,\sigma _{i-1},\pm ,\sigma _{i+1},...,\sigma _{n}}(t_{b})$. The
opening of the branches to the right means that, if the contact occurs from
the upper half-plane as $t$ increases, the corresponding local extremum is a
minimum, whereas if the contact occurs from the lower half-plane, it is a
maximum. In panel (d) of Figure \ref{figure4} we see that different branches
do not cross but touch at the bifurcation points (\textquotedblleft
T-crossings\textquotedblright ). We will show below that all these
properties hold in general.

\subsection{Smoothness domains of the root branches}\label{section32}

A crucial issue for our purposes is the distinction between $\mathbf{dom\,}\phi _{\sigma }$ and $\mathbf{sdom\,}\phi _{\sigma }$, the subset of $\mathbf{dom\,}\phi _{\sigma }$ where $\phi _{\sigma }(t)$ is smooth. As it
will turn out in Section \ref{section4}, $\mathbf{sdom\,}\phi _{\sigma }$ comprises the
parametric values $t$ for which the root $\phi _{\sigma }(t)$ is simple
---precisely the $t$'s that count for $s_{\left\vert \sigma \right\vert }(t)$, Equation (\ref{s^n(t)}). Therefore $\mathbf{sdom\,}\phi _{\sigma }$ can be
read not only as \textquotedblleft smoothness domain\textquotedblright\ but
also as \textquotedblleft simplicity domain\textquotedblright .

To learn about $\mathbf{sdom\,}\phi _{\sigma }$, we go back to the
neighborhood $U\subset \lbrack 0,2]$ of $t=2$ where the $2^{n}$ distinct
root branches $\phi _{\sigma _{1},...,\sigma _{n}}(t)$ are locally defined
and continuously differentiable. This neighborhood can be extended to
include lower and lower $t$ values as long as $\partial q_{t}^{n}(\phi
_{\sigma _{1},...,\sigma _{n}}(t))/\partial x\neq 0$, i.e., as long as $\phi
_{\sigma _{1},...,\sigma _{n}}(t)$ has not a vertical tangent. Since 
\begin{equation}
\frac{\partial q_{t}^{n}}{\partial x}(x)=\prod\limits_{k=0}^{n-1}\frac{dq_{t}}{dx}(q_{t}^{k}(x))=(-2)^{n}xq_{t}(x)\cdots q_{t}^{n-1}(x),  \label{Equat}
\end{equation}the obstruction $\partial q_{t}^{n}(\phi _{\sigma _{1},...,\sigma
_{n}}(t))/\partial x=0$ occurs whichever condition

\begin{description}
\item[(C)] $q_{t}^{k}(\phi _{\sigma _{1},...,\sigma _{n}}(t))=0$ ($0\leq
k\leq n-1$)
\end{description}

\noindent is fulfilled first. Conditions (\textbf{C}) comprise those
parametric values $t$ for which $\phi _{\sigma _{1},...,\sigma _{n}}(t)$ is
the critical point ($k=0$) and, for $n\geq 2$, any of its, at most $2^{n}-2$, preimages up to order $n-1$.

If $k=0$, then $q_{t}^{0}(\phi _{\sigma _{1},...,\sigma _{n}}(t))\equiv \phi
_{\sigma _{1},...,\sigma _{n}}(t)=0$. If $1\leq k\leq n-1$ ($n\geq 2$), then
use Equation (\ref{phi=x}) to derive\begin{equation}
q_{t}(\phi _{\sigma _{1},...,\sigma _{n}}(t))=t-\phi _{\sigma
_{1},...,\sigma _{n}}(t)^{2}=-\phi _{\sigma _{2},...,\sigma _{n}}(t)
\label{Ca}
\end{equation}and, in general, 
\begin{equation}
q_{t}^{k}(\phi _{\sigma _{1},...,\sigma _{n}}(t))=q_{t}(-\phi _{\sigma
_{k},...,\sigma _{n}}(t))=-\phi _{\sigma _{k+1},...,\sigma _{n}}(t),
\label{Cb}
\end{equation}so the conditions (\textbf{C}) amount to:

\begin{description}
\item[(C')] $\phi _{\sigma _{k+1},...,\sigma _{n}}(t)=0$ ($0\leq k\leq n-1$).
\end{description}

Note that\begin{equation}
\phi _{\sigma _{k+1},...,\sigma _{n}}(t_{b})=0\;\Rightarrow \;\phi _{-\sigma
_{k+1},\sigma _{k+2}...,\sigma _{n}}(t_{b})=-\phi _{\sigma _{k+1},\sigma
_{k+2}...,\sigma _{n}}(t_{b})=0,  \label{C''b}
\end{equation}therefore,\begin{equation}
\phi _{\sigma _{k+1},...,\sigma _{n}}(t_{b})=0\;\Rightarrow \;\left\{ 
\begin{array}{ll}
\phi _{\pm \sigma _{1},...,\sigma _{n}}(t_{b})=0 & \text{if }k=0, \\ 
\phi _{\sigma _{1},...,\pm \sigma _{k+1},...,\sigma _{n}}(t_{b})=\phi
_{\sigma _{1},...,\sigma _{k}}(t_{b}) & \text{if }1\leq k\leq n-1\text{,}\end{array}\right.  \label{C''c}
\end{equation}which means that $0$ ($k=0$) or $\phi _{\sigma _{1},...,\sigma _{k}}(t_{b})$
($1\leq k\leq n-1$) is a multiple zero of $q_{t_{b}}^{n}(x)$. Such a point $t_{b}$ is a \textit{branching} (or \textit{bifurcation}) \textit{point} of $\phi _{\sigma _{1},...,\pm \sigma _{k+1},...,\sigma _{n}}(t)$ if both
branches are defined in a neighborhood of $t_{b}$; otherwise, $t_{b}$ is an 
\textit{isolated point} of $\mathbf{dom\,}\phi _{\sigma _{1},...,\pm \sigma
_{k+1},...,\sigma _{n}}$. (Actually, one can check that the isolated points
of $\phi _{\sigma }(t)$, if any, correspond to branching points of some
predecessor.) Branching points and isolated points are called \textit{singular points}; the complement are the \textit{regular points} of the
corresponding root parabola or branches. This proves the following result:

\begin{Proposition}
\label{Propo40}The singular points $t_{b}$ of $\mathbf{dom\,}\phi _{\sigma }$
correspond to multiple zeros of $q_{t_{b}}^{n}(x)$. In either case, $\phi
_{\sigma _{k+1},...,\sigma _{n}}(t_{b})=0$ for some $k=0,1,...,n-1$.
\end{Proposition}

Furthermore, if $\left\vert \sigma \right\vert =s\geq r=\left\vert \rho
\right\vert $ and $\phi _{\sigma }(t_{0})=\phi _{\rho }(t_{0})$, i.e.,\begin{equation}
\sigma _{1}\sqrt{t_{0}+\phi _{\sigma _{2},...,\sigma _{s}}(t_{0})}=\rho _{1}\sqrt{t_{0}+\phi _{\rho _{2},...,\rho _{r}}(t_{0})},  \label{C''d}
\end{equation}then\begin{equation}
\phi _{\sigma }(t_{0})=\phi _{\rho }(t_{0})=0\;\;\text{if }\sigma _{1}\neq
\rho _{1};  \label{C''e}
\end{equation}otherwise, keep squaring the Equation (\ref{C''d}) and recursively applying
Equation (\ref{C''e}) to the resulting equalities to derive:\begin{equation}
\sigma _{i}\neq \rho _{i}\text{ for some }1\leq i\leq r\;\;\Rightarrow
\;\;\phi _{\sigma _{i},...,\sigma _{s}}(t_{0})=\phi _{\rho _{i},...,\rho
_{r}}(t_{0})=0,  \label{C''f}
\end{equation}or else\begin{equation}
\sigma _{i}=\rho _{i}\text{ for }1\leq i\leq r\text{ and }s>r\;\;\Rightarrow
\;\;\phi _{\sigma _{r+1},...,\sigma _{s}}(t_{0})=0.  \label{C''g}
\end{equation}By Proposition \ref{Propo40}, $t_{0}$ is a singular point of $\phi _{\sigma
}(t)$. We conclude:

\begin{Proposition}
\label{Propo4} A root branch $\phi _{\sigma }(t)$ can be smoothly extended
from the boundary $t=2$ to a maximal interval $\mathbf{sdom\,}\phi _{\sigma
}:=(t_{\sigma },2]$, where $t_{\sigma }=\max \{t\in \lbrack 0,2):\phi
_{\sigma _{k+1},...,\sigma _{n}}(t)=0$ for some $0\leq k\leq n-1\}$ is a
branching point of $\phi _{\sigma }(t)$. Moreover, $\phi _{\sigma }(t)\neq
\phi _{\rho }(t)$ for $\sigma \neq \rho $ and $t\in $ $\mathbf{sdom\,}\phi
_{\sigma }\cap \mathbf{sdom\,}\phi _{\rho }$.
\end{Proposition}

In other words, root branches do not cross or touch in their smoothness
domains. Table \ref{table1}, obtained from Figure \ref{figure4}, lists the smoothness
domains $(t_{\sigma },2]$ of the 15 root parabolas up to rank 4. 

\begin{table}[H]
	\caption{Root parabolas of ranks 1 to 4.}
	\label{table1}
	\centering
	%% \tablesize{} %% You can specify the fontsize here, e.g., \tablesize{\footnotesize}. If commented out \small will be used.
	\begin{tabular}{cc}
		\toprule
		\textbf{Root parabolas}	& $\mathbf{sdom\,}\phi _{\sigma }$\\
		\midrule
		$\phi _{\pm },\phi _{\pm +},\phi _{\pm ++},\phi _{\pm +++}$	& $(0,2]$	\\
		$\phi _{\pm -},\phi _{+\pm -},\phi _{-\pm -},\phi_{++\pm -},\phi _{-+\pm -},\phi _{\pm ---}$ & $(1,2]$	\\
		$\phi _{\pm -+-}$ & $(1.3107...,2]$  \\
		$\phi _{\pm -+},\phi _{+\pm -+},\phi _{-\pm -+}$ & $(1.7549...,2]$ \\ 
		$\phi _{\pm -++}$ & $(1.9408...,2]$ \\
		\bottomrule
	\end{tabular}
\end{table}

The ordering of the branching points $t_{\sigma }$ is related to the
ordering of the root branches. Due to the strictly increasing/decreasing
monotonicity of the positive/negative square root function, $\phi _{\sigma
}(t)<\phi _{\rho }(t)$ implies 
\begin{equation}
\phi _{-,\rho }(t)<\phi _{-,\sigma }(t)<\phi _{+,\sigma }(t)<\phi _{+,\rho
}(t).  \label{P1_1}
\end{equation}Thus, attaching a sign \textquotedblleft $+$\textquotedblright\ (resp.
\textquotedblleft $-$\textquotedblright ) in front the signature preserves
(resp. reverses) the ordering. This generalizes to the following \textit{signed lexicographical order} for root branches.

\begin{Proposition}
\label{Propo12} Given $\sigma \neq \rho $ with $\left\vert \sigma
\right\vert \geq \left\vert \rho \right\vert $ and $t\in $ $\mathbf{sdom\,}\phi _{\sigma }\cap \mathbf{sdom\,}\phi _{\rho }$, the following holds.

(a) If $\sigma _{1}\neq \rho _{1}$ then 
\begin{equation}
\phi _{\sigma }(t)\left\{ 
\begin{array}{cc}
>\phi _{\rho }(t) & \text{if }\sigma _{1}=+1, \\ 
<\phi _{\rho }(t) & \text{if }\sigma _{1}=-1.\end{array}\right.  \label{L3_1}
\end{equation}

(b) If $\sigma _{i}=\rho _{i}$ for $i=1,...,k$, and $k=\left\vert \rho
\right\vert $ or $\sigma _{k+1}\neq \rho _{k+1}$, then\begin{equation}
\phi _{\sigma }(t)\left\{ 
\begin{array}{cc}
>\phi _{\rho }(t) & \text{if }\sigma _{1}\times ...\times \sigma _{k+1}=+1,
\\ 
<\phi _{\rho }(t) & \text{if }\sigma _{1}\times ...\times \sigma _{k+1}=-1.\end{array}\right.  \label{L3_2}
\end{equation}
\end{Proposition}

Since root branches do not cross or touch in their smoothness domains, they
can be ordered alternatively by $\phi _{\sigma }(2)$. According to Equation (\ref{C''c}), the inequalities (\ref{L3_1})-(\ref{L3_2}) can turn equalities
at a common singular point of $\mathbf{dom\,}\phi _{\sigma }\cap \mathbf{dom\,}\phi _{\rho }$.

As an example,\begin{equation}
\phi _{-,\{+\}^{n-1}}(t)<\phi _{-,\sigma }(t)<\phi
_{-,-,\{+\}^{n-2}}(t)<\phi _{+,-,\{+\}^{n-2}}(t)<\phi _{+,\sigma }(t)<\phi
_{\{+\}^{n}}(t),  \label{P1_5}
\end{equation}for all $t\in \mathbf{sdom\,}\phi _{\pm ,\sigma }$, where $\left\vert \sigma
\right\vert =n-1$. Equation (\ref{P1_5}) shows the upper and lower bounds of
the positive and negative root branches.

\begin{Remark}
\label{Remark2} According to Equation (\ref{P1_5}),\begin{equation}
\lim_{n\rightarrow \infty }\phi _{\{+\}^{n}}(t)=\frac{1}{2}(1+\sqrt{1+4t})=:\phi _{\{+\}^{\infty }}(t)  \label{upper envelope}
\end{equation}and\begin{equation}
\lim_{n\rightarrow \infty }\phi _{-,\{+\}^{n-1}}(t)=-\lim_{n\rightarrow
\infty }\phi _{\{+\}^{n-1}}(t)=-\frac{1}{2}(1+\sqrt{1+4t})=:\phi
_{-,\{+\}^{\infty }}(t)  \label{lower envelope}
\end{equation}are the optimal upper and lower bounds, respectively, of all root branches
for $t>0$. From Equations (\ref{I_t}) and (\ref{invphi(x)}) we see that\begin{equation}
\lbrack \phi _{-,\{+\}^{\infty }}(t),\phi _{\{+\}^{\infty
}}(t)]=I_{t}=[-2\mu ,2\mu ],  \label{I_(t)2}
\end{equation}where $\mu =$ $\frac{1}{4}(1+\sqrt{1+4t})$ is the parameter value of the
logistic map $\left. f_{\mu }\right\vert _{[0,1]}$ conjugate to $\left.
q_{t}\right\vert _{I_{t}}$. Therefore, all zeros of $q_{t}^{n}(x)$ are in
the open interval $\mathring{I}_{t}=(-2\mu ,2\mu )=(-\frac{1}{2}(1+\sqrt{1+4t}),\frac{1}{2}(1+\sqrt{1+4t}))$ for $n\geq 1$.
\end{Remark}

\section{Application I: Monotonicity of the topological entropy}\label{section4}

In Section \ref{section3}, the smooth root branch $\phi _{\sigma }(t)$ was extended from
a neighborhood of the boundary $t=2$ to a maximal interval $\mathbf{sdom\,}\phi _{\sigma }=(t_{\sigma },2]$, called the smoothness domain of $\phi
_{\sigma }$. The next Proposition excludes the possibility that $\phi
_{\sigma }(t)$ is also defined in an interval $(t_{1},t_{2})$ with $0\leq
t_{1}<t_{2}\leq t_{\sigma }$. By the same arguments used with $\mathbf{sdom\,}\phi _{\sigma }$, the endpoints $t_{1}$ and $t_{2}$ would be then branching
points of $\phi _{\sigma }(t)$.

\begin{figure}[tbp]
\centering
\includegraphics[height=3cm]{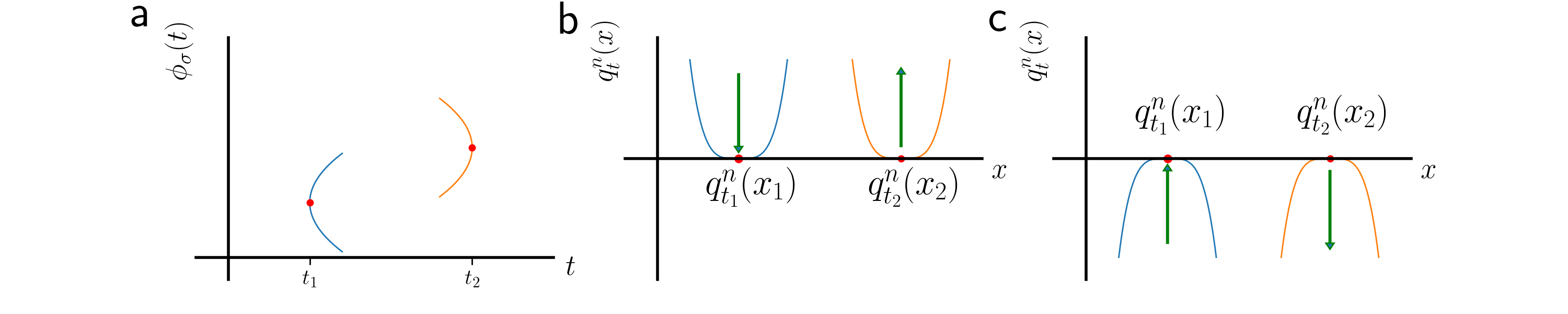}
\caption{(a) Root branches not connected to the boundary $t=2$. As a new
feature, the bifurcation at $t_{2}$ opens to the left. (b) A first
possibility in state space for local extrema of $q_{t}^{n}(x)$ to produce
the bifurcations in panel (a). (c) A second possibility in state space for
local extrema of $q_{t}^{n}(x)$ to produce the bifurcations in panel (a).}
\label{figure5}
\end{figure}

\begin{Proposition}
\label{Propo5} For all $\left\vert \sigma \right\vert \geq 1$, $\mathbf{dom\,}\phi _{\sigma }$ does not include intervals other than $[t_{\sigma },2].$
\end{Proposition}

\begin{proof}
Suppose that the, say positive, root branch $\phi _{\sigma _{1},...,\sigma
_{n}}(t)$ is also defined in an interval $(t_{1},t_{2})\subset \lbrack 0,2]$, where $0\leq t_{1},t_{2}\leq t_{\sigma }$ are two branching points, hence, 
$\phi _{\sigma _{k_{1}+1},...,\sigma _{n}}(t_{1})=\phi _{\sigma
_{k_{2}+1},...,\sigma _{n}}(t_{2})=0$ for some $0\leq k_{1},k_{2}\leq n-1$
(Proposition \ref{Propo4}). In this case, the positive root branches $\phi
_{\sigma _{1},...,\pm \sigma _{k_{1}+1},...,\sigma _{n}}(t)$ and $\phi
_{\sigma _{1},...,\pm \sigma _{k_{2}+1},...,\sigma _{n}}(t)$ would compose
the two parabolas depicted in Figure 5(a) in a neighborhood of $t_{1}$ and $t_{2}$, respectively. 

The \textquotedblleft $\subset $\textquotedblright\ bifurcation at
\textquotedblleft time\textquotedblright\ $t_{1}$ corresponds to a local
minimum (resp. local maximum) of $q_{t}^{n}(x)$ crossing the $x$-axis from
above (resp. below) at the point $x_{1}=\phi _{\sigma _{1},...,\sigma
_{n}}(t_{1})=0$ if $k_{1}=0$ or $x_{1}=\phi _{\sigma _{1},...,\sigma
_{k_{1}}}(t_{1})>0$ if $k_{1}\geq 1$ ( see Figure \ref{figure5}(b)-(c) and
Equation (\ref{C''c})). Bifurcation points with branches opening to the
right occur at the left endpoint of the smoothness domains, in particular at 
$t=0$, so they are certainly allowed.

The \textquotedblleft $\supset $\textquotedblright\ bifurcation at
\textquotedblleft time\textquotedblright\ $t_{2}$ corresponds to local a
minimum (resp. maximum) of $q_{t}^{n}(x)$ crossing the $x$-axis from below
(resp. above) at the point $x_{2}=\phi _{\sigma _{1},...,\sigma
_{n}}(t_{2})=0$ if $k_{2}=0$ or $x_{2}=\phi _{\sigma _{1},...,\sigma
_{k_{2}}}(t_{2})>0$ (see Figure \ref{figure5}(b)-(c) and Equation (\ref{C''c})). To show that bifurcation points with branches opening to the left,
however, are not allowed, we are going to exploit the following \emph{Fact}
derived from the hypothetical existence of $\supset $\emph{\ }bifurcations.

\emph{Fact}: In both cases illustrated in Figure \ref{figure5}(b) (where $q_{t_{2}}^{n}(x_{2})$ is a local minimum and $\left. \partial
q_{t}^{n}(x_{2})/\partial t\right\vert _{t=t_{2}}\geq 0$) and Figure \ref{figure5}(c) (where $q_{t_{2}}^{n}(x_{2})$ is a local maximum $\left.
\partial q_{t}^{n}(x_{2})/\partial t\right\vert _{t=t_{2}}\leq 0$), given
any neighborhood of $x_{2}$, $(x_{2}-\varepsilon ,\,x_{2}+\varepsilon )$
with $\varepsilon >0$, there exists $\tau >0$ such that $q_{t_{2}-\bigtriangleup t}^{n}(x)$ changes sign in $(x_{2}-\varepsilon
,\,x_{2}+\varepsilon )$ for all $0<$ $\bigtriangleup t\leq \tau $ because,
by assumption, $q_{t}^{n}(x)$ intersects transversally the $x$-axis just
before $t=t_{2}.$

It is even more true: said change of sign occurs both in $(x_{2}-\varepsilon
,\,x_{2})$ due to the left branch, and in $(x_{2},\,x_{2}+\varepsilon )$ due
to the right branch. Note that the above \emph{Fact} does not hold for $\subset $\emph{\ }bifurcations.

Therefore, we will consider only the first case (Figure \ref{figure5}(b)
with $x=x_{2}$ and $t=t_{2}$). There are several subcases.

(a) If $\left. \partial q_{t}^{n}(x_{2})/\partial t\right\vert _{t=t_{2}}>0$, then 
\begin{equation}
q_{t_{2}-\bigtriangleup t}^{n}(x_{2}\pm \bigtriangleup x)=-\partial
_{t}q_{t_{2}}^{n}(x_{2})\cdot \bigtriangleup t+O(2),  \label{Eq(a)}
\end{equation}where $0<\bigtriangleup x\leq \varepsilon $, $0<\bigtriangleup t\leq \tau $,
and we wrote $\partial _{t}q_{t_{2}}^{n}(x_{2})$ $\equiv \left. \partial
_{t}q_{t}^{n}(x_{2})\right\vert _{t=t_{2}}$ for brevity. From Equation (\ref{Eq(a)}) it follows $q_{t_{i}-\bigtriangleup t}^{n}(x)<0$ in $(x_{2}-\varepsilon ,\,x_{2}+\varepsilon )$ for all $\bigtriangleup t$, which
contradicts the above \emph{Fact}. This excludes the possibility of having a $\supset $ bifurcation at $t_{2}$ if
the \textquotedblleft velocity\textquotedblright\ of $q_{t}^{n}(x_{2})$ at $t=t_{2}$ is positive.

(b) Suppose now $\partial q_{t_{2}}^{n}(x_{2})/\partial t=0$, so\begin{equation}
q_{t_{2}-\Delta t}^{n}(x_{2}\pm \bigtriangleup x)=\frac{1}{2}\partial
_{xx}q_{t_{2}}^{n}(x_{2})(\bigtriangleup x)^{2}\mp \partial
_{xt}^{2}q_{t_{2}}^{n}(x_{2})\bigtriangleup x\bigtriangleup t+\frac{1}{2}\partial _{tt}q_{t_{2}}^{n}(x_{2})(\bigtriangleup t)^{2}+O(3),  \label{Eq(b)}
\end{equation}where $\partial _{xx}q_{t_{2}}^{n}(x_{2})\equiv \left. \partial
_{xx}q_{t_{2}}^{n}(x)\right\vert _{x=x_{2}}\geq 0$ (because $q_{t_{2}}^{n}(x_{2})$ is a minimum in the case we are considering), $\partial _{tt}q_{t_{2}}^{n}(x_{2})\equiv \left. \partial
_{tt}q_{t}^{n}(x_{2})\right\vert _{t=t_{2}}$ and similarly for the mixed
term.

(b1) If $\partial _{xx}q_{t_{2}}^{n}(x_{2})>0$ and some of the other $O(2)$
terms is not zero, let $\bigtriangleup t\rightarrow 0$ while $\bigtriangleup
x$ is fixed to conclude from Equation (\ref{Eq(b)}) that $q_{t_{2}-\Delta
t}^{n}(x_{2}\pm \bigtriangleup x)$ does not change sign for sufficiently
small $\bigtriangleup t$, $0<\bigtriangleup t\ll \bigtriangleup x$, in
contradiction to the above \textit{Fact}. The same contraction follows, of
course, if all $O(2)$ terms in Equation (\ref{Eq(b)}) except $\partial
_{xx}q_{t_{2}}^{n}(x_{2})$ vanish.

(b2) If all terms $O(2)$ vanish at $x=x_{2}$ and $t=t_{2}$, repeat the same
argument with the terms $O(3)$. Since $q_{t_{2}}^{n}(x_{2})$ is a minimum
and $q_{t}^{n}(x)$ is a polynomial of degree $2^{n}$, it holds $\partial
_{x^m}^{m}q_{t_{2}}^{n}(x_{2})>0$ for some $2\leq m\leq 2^{n}$.
\end{proof}

A conclusion of the proof of Proposition \ref{Propo5} is that the root
branches do not have bifurcations with branches that open to the left or
bifurcations that open to the right except for the one at the left endpoint
of the smoothness domain. As a result:

\begin{Proposition}
\label{Propo5B}For all $\left\vert \sigma \right\vert \geq 1$, $\mathbf{sdom\,}\phi _{\sigma }=(t_{\sigma },2]$, where $0\leq t_{\sigma }<2$ is the
unique branching point of $\phi _{\sigma }$.
\end{Proposition}

\begin{Remark}
\label{Remark1} The images of the critical point build a sequence of
polynomials $P_{n}(t):=q_{t}^{n}(0)$, that is, $P_{n}(t)$ is the $n$th image
of $0$ under $q_{t}$. Alternatively, one can define polynomial maps $P_{n}:[0,2]\rightarrow \lbrack -2,2]$ by the recursion 
\begin{equation}
P_{0}(t)=0,\;\;P_{n}(t)=t-P_{n-1}(t)^{2}\text{\ \ for }n\geq 1\text{.}
\label{psi(t)}
\end{equation}Therefore, $P_{n}(t)$ is a polynomial of degree $2^{n-1}$ for $n\geq 1$, and\begin{equation}
P_{n+k}(t)=q_{t}^{n+k}(0)=q_{t}^{n}(q_{t}^{k}(0))=q_{t}^{n}(P_{k}(t)).
\label{psi(t)B}
\end{equation}The first polynomials are:\begin{equation}
\begin{array}{l}
P_{1}(t)=t \\ 
P_{2}(t)=t-t^{2} \\ 
P_{3}(t)=t-t^{2}+2t^{3}-t^{4} \\ 
P_{4}(t)=t-t^{2}+2t^{3}-5t^{4}+6t^{5}-6t^{6}+4t^{7}-t^{8}\end{array}
\label{phi(t)}
\end{equation}If, as in the proof of Proposition \ref{Propo5}, we interpret the parameter $t$ as time, then the time of passage of $q_{t}^{n}(0)$ through the $x$-axis
is given by the zeros of $P_{n}(t)=0$. Note that\begin{equation}
P_{n}(0)=0\;\text{ for }n\geq 1,  \label{psi(0)}
\end{equation}while\begin{equation}
P_{1}(2)=2,\;P_{n}(2)=-2\text{ \ for }n\geq 2  \label{phi(t)B}
\end{equation}(see Figure \ref{figure2}(d) for $n=4$). In physical terms, $q_{t}(0)$ moves
upwards from $x=0$ ($t=0$) to $x=2$ ($t=2$) at constant speed $\dot{P}_{1}(t)=1$ (the dot denotes time derivative), while, for $n\geq 2$, $q_{t}^{n}(0)$ moves from $x=0$ ($t=0$) to $x=-2$ ($t=2$), reversing the
speed when $\dot{P}_{n}(t)=0$ and crossing the $x$-axis when $P_{n}(t)=0$.
In Section \ref{section5} we will come back to these polynomials from a different
perspective.
\end{Remark}

At this point we have already cleared our way to the monotonicity of the
topological entropy for the quadratic family,\begin{equation}
h(q_{t})=\lim_{n\rightarrow \infty }\frac{1}{n}\log \left(
1+\sum\limits_{k=0}^{n-1}s_{k}(t)\right) ,  \label{h(q_t)}
\end{equation}where $s_{k}(t)$ is the number of \textit{simple} zeros of $q_{t}^{k}(x)$
or, equivalently, the number of \textit{transverse} intersections of the
curve $y=q_{t}^{k}(x)$ with the critical line $y=0$ (the $x$-axis); see
Equation (\ref{s^n(t)}). According to Equation (\ref{P1_5}) and Remark \ref{Remark2}, all zeros of $q_{t}^{k}(x)$ are in the interval $[\phi
_{-,\{+\}^{k-1}}(t),\phi _{\{+\}^{k}}(t)]$ $\subset (\phi _{-,\{+\}^{\infty
}}(t),\phi _{\{+\}^{\infty }}(t))=\mathring{I}_{t}$.

It follows from Propositions \ref{Propo40} and \ref{Propo5B}, that, for each
signature $\sigma $ with $\left\vert \sigma \right\vert =n$, $\mathbf{dom}\,\phi _{\sigma }\backslash \mathbf{sdom\,}\phi _{\sigma }$ comprises
multiple roots of $q_{t}^{n}(t)$ (isolated points and the branching point $t_{\sigma }$), while the roots $\phi _{\sigma }(t)$ are simple for $t_{\sigma }<t\leq 2$ by Proposition \ref{Propo4}. The bottom line is:

\begin{Proposition}
\label{Propo5C}The smoothness domain $\mathbf{sdom\,}\phi _{\sigma }$
comprises the values of $t$ for which the root $\phi _{\sigma }(t)$ of $q_{t}^{\left\vert \sigma \right\vert }(x)$ is simple.
\end{Proposition}

For this reason we anticipated at the beginning of Section \ref{section32} that $\mathbf{sdom\,}\phi _{\sigma }$ may be called the \textit{simplicity domain} of $\phi _{\sigma }$ as well. This being the case, each root $\phi _{\sigma }(t)$
contributes 0 or 1 to $s_{n}(t)$, the number of simple zeros of $q_{t}^{n}(x) $, depending on whether $0<t\leq t_{\sigma }$ or $t_{\sigma
}<t\leq 2$, respectively. We conclude that\begin{equation}
s_{n}(t)=\sum_{\sigma \in \{+,-\}^{n}}\chi _{(t_{\sigma },2]}(t)=2\sum_{\rho
\in \{+,-\}^{n-1}}\chi _{(t_{+,\rho },2]}(t)  \label{charact}
\end{equation}for $n\geq 1$, where we used $t_{-\sigma _{1},\sigma _{2},...,\sigma
_{n}}=t_{\sigma _{1},\sigma _{2},...,\sigma _{n}}$, and $\chi _{(t_{\sigma
},2]}(t)$ is the characteristic or indicator function of the interval $(t_{\sigma },2]$ ($1$ if $t\in (t_{\sigma },2]$, 0 otherwise). Equation (\ref{charact}) proves:

\begin{Theorem}
\label{Theo8} The function $s_{n}:[0,2]$ $\rightarrow \{0,2,4,...,2^{n}\}$,
defined in Equation (\ref{s^n(t)}), is piecewise constant and monotonically increasing for every $n\geq 1$. Its discontinuities occur at the branching
points $t_{\sigma }$ of the root branches $\phi _{\sigma }(t)$ with $\left\vert \sigma \right\vert =n$, where $s_{n}(t)$ is lower semicontinuous.
\end{Theorem}

Figure \ref{figure6} shows the function $s_{4}(t)$ based on Figure \ref{figure4}(c). Apply now Proposition \ref{Propo3} to prove Milnor's
Monotonicity Conjecture for the quadratic family:

\begin{Theorem}
\label{Theo9} The topological entropy of $q_{t}$ is a monotonically
increasing function of $t$.
\end{Theorem}

Figure \ref{figure7} shows the topological entropy of $q_{t}$ superimposed
on the bifurcation diagram (Figure \ref{figure3}); in particular, $h(q_{t})>0 $ (i.e., $q_{t}$ is chaotic) for $t>t_{F}=1.4011551...$
(Feigenbaum point) and $h(q_{2})=\log 2$, the highest value that the
topological entropy of a unimodal map can take, see Equation (\ref{maxh(f)}). It can be proved that the function $t\mapsto h(q_{t})$ is a Devil's
staircase, meaning that it is continuous, monotonically increasing (Theorem \ref{Theo9}), but there is no interval of parameters where it is strictly
increasing \cite{Graczyk1997,Lyubich1997}. The plateaus where $h(q_{t})$ is
constant correspond to intervals containing a periodic attractor and the
subsequent period doubling cascade (e.g., the period-3 window, clearly
visible in Figure \ref{figure7}). This shows that periodic orbits do not
disappear as $t $ increases, however, the new periodic orbits that are
created do not necessarily increase $h(q_{t})$. The topological entropy in
Figure \ref{figure7} was computed using the general algorithm presented in 
\cite{Amigo2015}, but see \cite{Dilao2012} for a simpler and quicker
algorithm adapted to unimodal maps. The small but positive values of $h(q_{t})$ to the left of $t_{F}$ are due to the slow convergence of the
algorithm.

\begin{figure}[tbp]
\centering
\includegraphics[height=4cm]{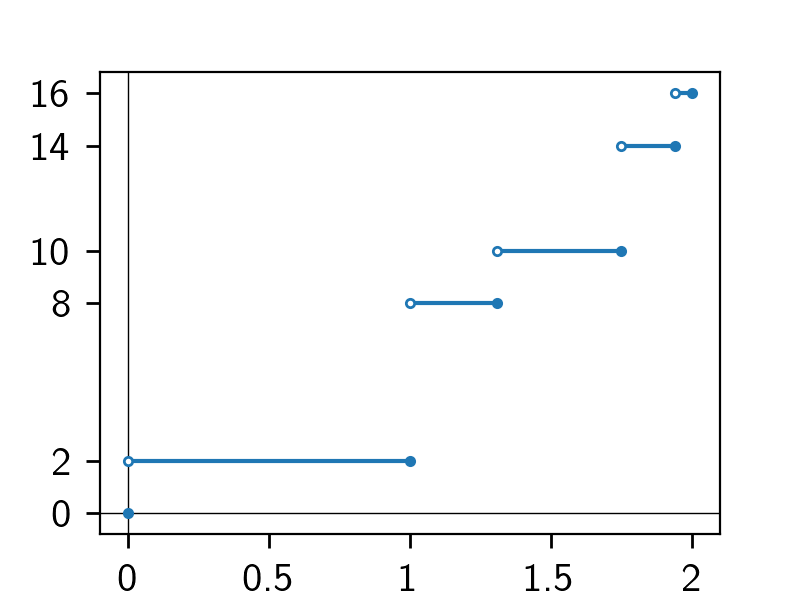}
\caption{The function $s_{4}(t)$ for the quadratic family. Jumps occur at
the branching points $t_{++++}=0$, $t_{+---}=t_{++--}=t_{+++-}=1$, $t_{+-+-}\simeq 1.3107$, $t_{+--+}=t_{++-+}\simeq 1.7549$, and $t_{+-++}\simeq 1.9408$.}
\label{figure6}
\end{figure}

\begin{figure}[tbp]
\centering
\includegraphics[height=4cm]{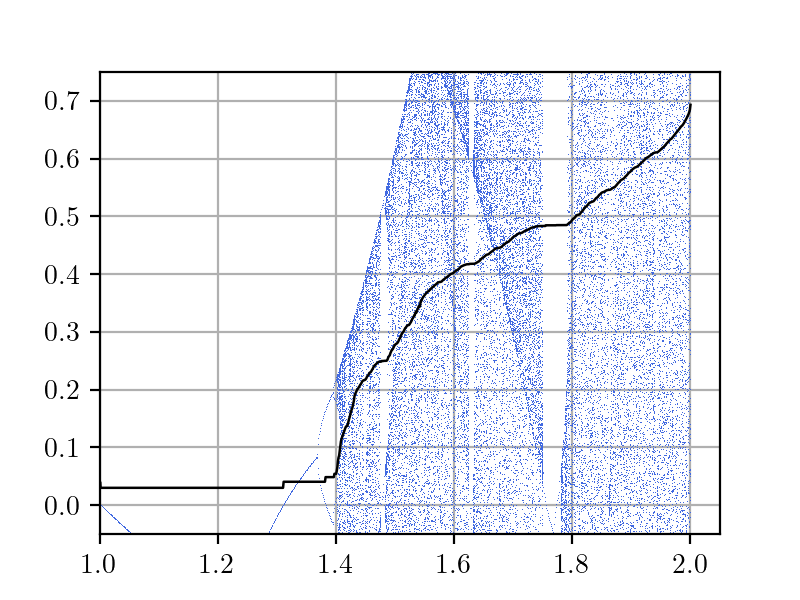}
\caption{Topological entropy of $q_{t}$ using the formula (\protect\ref{h(q_t)}) with the logarithm to the base $e$. The topological entropy was
plotted on the bifurcation diagram for a better understanding of its
characteristics.}
\label{figure7}
\end{figure}

\section{Application II: Superstable period orbits}\label{section5}

In Section \ref{section4} we studied the solutions of the equation $q_{t}^{n}(x)=0$,
where the parameter $t$ was thought to be fixed. In other words, we were
interested in the zeros of a polynomial function of the variable $x$ and,
more particularly, in the values of $t$ for which those zeros were simple.
If we fix $x$ instead, then the solutions of $q_{t}^{n}(x)=0$ are the
parametric values $t$ such that $x$ is an $n$-order preimage of $0$, which
is the critical point of $q_{t}$. If, moreover, $x=0$, then the solutions
are the parametric values $t$ for which the critical value is periodic with
period $n$. As explained at the beginning of Section \ref{section23}, these orbits are
called \textit{superstable} because then the derivative of $q_{t}^{n}$ at
each point of the periodic orbit is 0 (see Equation (\ref{super})). For the
quadratic maps, $q_{t}(0)=t$, so $0$ is not a fixed point for $t>0$.

\begin{Remark}
\label{Remark3}If in Section \ref{section4} our main concern were the transverse
intersections of $q_{t}^{n}(x)$ with the $x$-axis, in this section it will
be the transverse intersection of the bisector with the positive root
branches (if any). In this regard, note that the bisector can intersect a
positive root branch $\phi _{+,\sigma }(t)$ at a regular point $t_{0}$
(i.e., $t_{0}\in (t_{+,\sigma },2]=\mathbf{sdom\,}\phi _{+,\sigma }$) only
once and transversally, from above to below. Otherwise, the root parabola $\phi _{\pm ,-,\sigma }(t)=\pm \sqrt{t-\phi _{+,\sigma }(t)}$ would have
multiple branching points, contradicting Proposition \ref{Propo5B}. Singular
points are isolated or at the boundary of smoothness domains (branching
points), so the concept of transversal intersection do not apply to them.
\end{Remark}

\subsection{Symbolic sequences}\label{section51}

To study the superstable periodic orbits of $q_{t}$, it suffices to consider
symbolic orbits. As an advantage, the results hold also under
order-preserving conjugacies, as happens with $\left. q_{t}\right\vert
_{[-2\mu ,2\mu ]}$ and $\left. f_{\mu }\right\vert _{[0,1]}$ under the
affine transformation (\ref{invphi(x)}). We come back to this point below.

Given a general orbit $(q_{t}^{k}(x_{0}))_{k=0}^{\infty
}=(x_{0},x_{1},...,x_{k},...)$, the corresponding symbolic sequence $\Sigma
=(\Sigma _{0},\Sigma _{1},...,\Sigma _{k},...)$ is defined as follows:\begin{equation*}
\Sigma _{k}=\left\{ 
\begin{array}{cc}
- & \text{if }x_{k}<0 \\ 
C & \text{if }x_{k}=0 \\ 
+ & \text{if }x_{k}>0\end{array}\right.
\end{equation*}for $k\geq 0$. A symbolic sequence that corresponds to an actual orbit of $q_{t}$ for $t=t_{0}$ is called \textit{admissible} (for $t=t_{0}$) and have
to fulfill certain conditions \cite{Hao1998}.

Consider a superstable periodic orbit $(0,x_{1},x_{2},...,x_{p-1})^{\infty }$
of \textit{prime} period $p\geq 2$, so that $x_{k}\neq 0$ for $k=1,...,p-1$.
For the time being, we drop the exponent \textquotedblleft $\infty $\textquotedblright\ and rearrange the cycle as $(x_{1},x_{2},...,x_{p-1},0)$
so that the first point is the critical value (also the greatest value) $q_{t}(0)=t>0$. In this case, $\Sigma _{k}\in \{+,-\}$ for $1\leq k\leq p-1$,
so that we will fittingly use $\sigma $'s instead of $\Sigma $'s and write
the pertaining symbolic sequence as $(+,\sigma _{2},...,\sigma _{p-1},C)$.
Therefore, by writing $\Sigma =(+,\sigma _{2},...,\sigma _{p-1},C)$ we do
not need to specify that $\Sigma $ is a superstable cycle of prime period $p$. The parameter values for which $q_{t}$ has superstable cycles are discrete
because $q_{t}^{n}(0)=0$ is a polynomial equation in $t$ for each $n$ (see
Remark \ref{Remark1}); we will see below that there are infinitely many such
values and that they accumulate at the right endpoint of the parametric
interval, $t_{\max }=2$.

\begin{Proposition}
\label{Propo8} Let $\Sigma =(+,\sigma _{2},...,\sigma _{p-1},C)$ be an
admissible cycle for $t=t_{0}$. Then $t_{0}$ satisfies the equation\begin{equation}
t_{0}=\phi _{+,-\sigma _{2},...,-\sigma _{p-1}}(t_{0}).  \label{Propo8a}
\end{equation}Equivalently,\begin{equation}
\phi _{\pm ,-,-\sigma _{2},...,-\sigma _{p-1}}(t_{0})=0.  \label{Propo8b}
\end{equation}Moreover, $t_{0}$ is a regular point of $\phi _{+,-\sigma _{2},...,-\sigma
_{p-1}}(t_{0})$, therefore $t_{0}$ is the branching point of the on-axis
root parabola $\phi _{\pm ,-,-\sigma _{2},...,-\sigma _{p-1}}(t)$.
\end{Proposition}

\begin{proof}
From (i) $x_{1}=q_{t_{0}}(0)=t_{0}$, (ii) $x_{k+1}=q_{t_{0}}(x_{k})=t_{0}-x_{k}^{2}$, i.e.,\begin{equation*}
x_{k}=\sigma _{k}\sqrt{t_{0}-x_{k+1}}\text{\ \ for\ }k=1,2,...,p-1,
\end{equation*}and (iii) $x_{p}=0$, we obtain 
\begin{equation*}
t_{0}=\sqrt{t_{0}-x_{2}}=\sqrt{t_{0}-\sigma _{2}\sqrt{t_{0}-x_{3}}}=...=\sqrt{t_{0}-\sigma _{2}\sqrt{t_{0}-...-\sigma _{p-1}\sqrt{t_{0}}}}=\phi
_{+,-\sigma _{2},...,-\sigma _{p-1}}(t_{0}).
\end{equation*}Alternatively,\begin{equation*}
t_{0}=\phi _{+,-\sigma _{2},...,-\sigma _{p-1}}(t_{0})\;\Leftrightarrow
\;\;\pm \sqrt{t_{0}-\phi _{+,-\sigma _{2},...,-\sigma _{p-1}}(t_{0})}\equiv
\phi _{\pm ,-,-\sigma _{2},...,-\sigma _{p-1}}(t_{0})=0.
\end{equation*}

Moreover, if $t_{0}$ is a singular point of $\phi _{+,-\sigma
_{2},...,-\sigma _{p-1}}(t_{0})$, then by Equation (\ref{C''c}) and $t_{0}>0$, it holds $\phi _{-\sigma _{k+1},...,-\sigma _{p-1}}(t_{0})=0$ for some $1\leq k\leq p-2$, so that\begin{equation*}
t_{0}=\phi _{+,-\sigma _{2},...,-\sigma _{p-1}}(t_{0})=\phi _{+,-\sigma
_{2},...,-\sigma _{k}}(t_{0}).
\end{equation*}It follows,\begin{equation*}
x_{k+1}=q_{t_{0}}^{k+1}(0)=q_{t_{0}}^{k}(t_{0})=q_{t_{0}}^{k}(\phi
_{+,-\sigma _{2},...,-\sigma _{k}}(t_{0}))=0
\end{equation*}by definition of root branches of rank $k$, which contradicts that $\Sigma
_{k+1}\neq C$ ($k+1\leq p-1$).

As explained in Remark \ref{Remark3}, root branches have at regular points transversal
intersections (if any) with the bisector. This implies that the branches of
the root parabola $\phi _{\pm ,-,-\sigma _{2},...,-\sigma _{p-1}}(t)$ are
defined in a neighborhood of $t_{0}$, therefore $t_{0}$ is a branching point
of $\phi _{\pm ,-,-\sigma _{2},...,-\sigma _{p-1}}(t)$.
\end{proof}

By the same token, if there exists no solution of the equation $t=\phi
_{+,-\sigma _{2},...,-\sigma _{p-1}}(t)$, then the cycle $(+,\sigma
_{2},...,\sigma _{p-1},C)$ is not admissible. So, root branches and their
bifurcation points, Equations (\ref{Propo8a}) and (\ref{Propo8b}), pop up as
soon as one learns about superstable cycles. Next we prove the reverse
implication of Proposition \ref{Propo8}.

\begin{Proposition}
\label{Propo9} If the bisector intersects transversally the root branch $\phi _{+,\sigma _{2},...,\sigma _{p-1}}(t)$ at $t_{0}$, then $(+,-\sigma
_{2},...,-\sigma _{p-1},C)$ is an admissible cycle for $t=t_{0}$.
\end{Proposition}

\begin{proof}
Suppose $q_{t_{0}}(0)=t_{0}=\phi _{+,\sigma _{2},...,\sigma _{p-1}}(t_{0})$.
Then, similarly to Equations (\ref{Ca})-(\ref{Cb}),\begin{eqnarray*}
q_{t_{0}}^{2}(0) &=&q_{t_{0}}(\phi _{+,\sigma _{2},...,\sigma
_{p-1}}(t_{0}))=-\phi _{\sigma _{2},...,\sigma _{p-1}}(t_{0})=\phi _{-\sigma
_{2},\sigma _{3},...,\sigma _{p-1}}(t_{0}),... \\
q_{t_{0}}^{k}(0) &=&q_{t_{0}}(\phi _{-\sigma _{k-1},\sigma _{k},...,\sigma
_{p-1}}(t_{0}))=-\phi _{\sigma _{k},...,\sigma _{p-1}}(t_{0})=\phi _{-\sigma
_{k},\sigma _{k+1},...,\sigma _{p-1}}(t_{0})
\end{eqnarray*}for $k=3,...,p-1$, and\begin{equation*}
q_{t_{0}}^{p}(0)=q_{t_{0}}(\phi _{-\sigma _{p-1}}(t_{0}))=t_{0}-\phi
_{-\sigma _{p-1}}(t_{0})^{2}=t_{0}-(-\sigma _{p-1}\sqrt{t_{0}})^{2}=0.
\end{equation*}It follows that the points\begin{equation*}
(\phi _{+,\sigma _{2},...,\sigma _{p-1}}(t_{0}),\;\phi _{-\sigma _{2},\sigma
_{3},...,\sigma _{p-1}}(t_{0}),\;\phi _{-\sigma _{3},\sigma _{4},...,\sigma
_{p-1}}(t_{0}),\;...,\;\phi _{-\sigma _{p-1}}(t_{0}),\;0)
\end{equation*}build a periodic orbit. Its symbolic sequence $\Sigma =(+,\Sigma
_{2},...,\Sigma _{p-1},C)$ is determined by the signs of $\phi _{-\sigma
_{k},\sigma _{k+1},...,\sigma _{p-1}}(t_{0})$ for $2\leq k\leq p-1$. Since $t_{0}$ is necessarily a regular point of $\phi _{+,\sigma _{2},...,\sigma
_{p-1}}(t)$ (Remark \ref{Remark3}), it holds $\phi _{-\sigma _{k},\sigma
_{k+1},...,\sigma _{p-1}}(t_{0})\neq 0$ for $2\leq k\leq p-1$, and hence $\Sigma _{k}=-\sigma _{k}\in \{+,-\}.$
\end{proof}

Table \ref{table2} lists the superstable cycles of $q_{t}$ of prime periods $p=1,2,...6$
(cycles written in abridged notation). The cycles of periods 2 and 3 are due
to the transverse intersections of the bisector with $\phi _{+}(t)=\sqrt{t}$
at $t=1$ and with $\phi _{++}(t)=\sqrt{t+\sqrt{t}}$ at $t\simeq 1.7549$.
Figure \ref{figure8} illustrates where the superstable cycles of prime
periods $4$ and $5 $ come from.

\begin{table}[H]
	\caption{Superstable cycles of periods 1 to 6.}
	\label{table2}
	\centering
	%% \tablesize{} %% You can specify the fontsize here, e.g., \tablesize{\footnotesize}. If commented out \small will be used.
	\begin{tabular}{cc}
		\toprule
		\textbf{Period} & \textbf{Superstable cycles}\\
		\midrule
		1 & $C$ \\ \hline
		2 & $+C$ \\ \hline
		3 & $+-C$ \\ \hline
		4 & $+-+C,\;+\{-\}^{2}C$ \\ \hline
		5 & $+-\{+\}^{2}C,$ $+\{-\}^{2}+C,$ $+\{-\}^{3}C$ \\ \hline
		6 & $+-\{+\}^{3}C,\;+\{-\}^{2}+-C,\;+\{-\}^{2}\{+\}^{2}C,\;+\{-\}^{3}+C,\;+\{-\}^{4}C$ \\
		\bottomrule
	\end{tabular}
\end{table}

\begin{figure}[tbp]
\centering
\includegraphics[height=3cm]{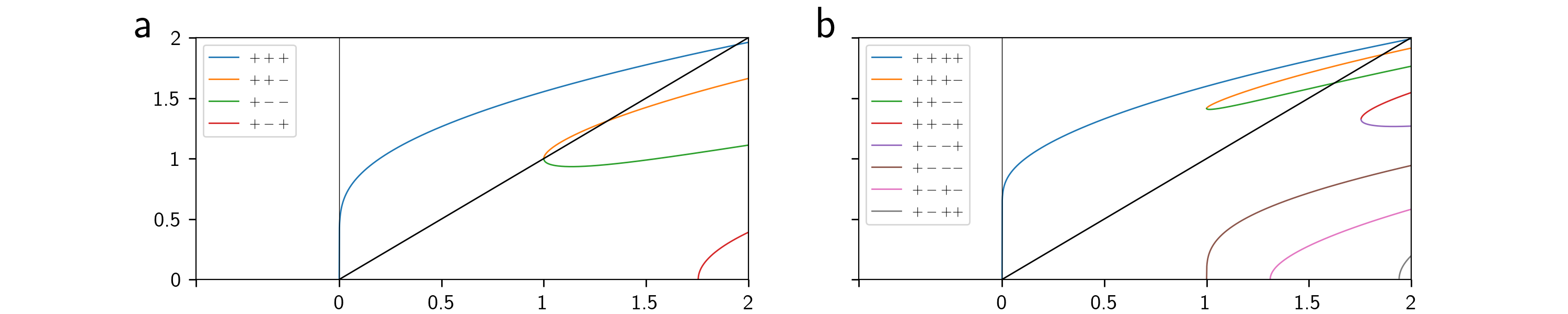}
\caption{(a) Non-transverse intersections (\textquotedblleft
T-crossings\textquotedblright ) of the bisector with $\protect\phi _{++-}(t)$
and $\protect\phi _{+--}(t)$ at $t=1$, and transverse intersections of the
bisector with (bottom to top) $\protect\phi _{++-}(t)$ at $t\simeq 1.3107$
and $\protect\phi _{+++}(t)$ at $t\simeq 1.9408$. Proposition \protect\ref{Propo9} entails that $q_{1.3107...}$ has the cycle $(+,-,+,C)$ and $q_{1.9408...}$ the cycle $(+,-,-,C)$, whereas $q_{1}$ has the cycle $(+,C)$.
(b) Regular intersections of the bisector with (bottom to top) $\protect\phi _{++--}(t)$ at $t\simeq 1.6254$, $\protect\phi _{+++-}(t)$ at $t\simeq
1.8607 $, and $\protect\phi _{++++}(t)$ at $t\simeq 1.9854$. Proposition 
\protect\ref{Propo9} entails that $q_{1.6254...}$, $q_{1.8607...}$ and $q_{1.9854...} $ have the cycles $(+,-,+,+,C)$, $(+,-,-,+,C)$, and $(+,-,-,-,C) $, respectively.}
\label{figure8}
\end{figure}

\begin{Proposition}
\label{Propo10} The quadratic family has superstable cycles of arbitrary
length.
\end{Proposition}

\begin{proof}
First, $\mathbf{sdom\,}\phi _{\{+\}^{n}}=(0,2]$ for all $n\geq 1$, so all $t>0$ are regular points of $\phi _{\{+\}^{n}}(t)$. Second, the bisector and $\phi _{\{+\}^{n}}(t)$ always intersect transversally at a single point $t_{n}^{\ast }\in \lbrack 1,2)$ because (i) $\phi _{\{+\}^{n}}(t)>t$ for $0<t<1$, (ii) $\phi _{\{+\}^{n}}(2)<\phi _{\{+\}^{\infty }}(2)=2$ (see
Equation (\ref{upper envelope})) and (iii) $\phi _{\{+\}^{n}}(t)$ is $\cap $-convex. Moreover, $t_{n}^{\ast }\rightarrow 2$ strictly monotonically as $n\rightarrow \infty $ because $\phi _{\{+\}^{n}}(t)<\phi _{\{+\}^{n+1}}(t)$
for all $n\geq 1$. The point $t_{n}^{\ast }$ is the branching point of the
root parabola $\phi _{\pm ,-,\{+\}^{n-1}}(t)$, i.e., $t_{n}^{\ast }=$ $t_{\pm ,-,\{+\}^{n-1}}$ in the notation of Sections \ref{section3} and \ref{section4}.
\end{proof}

According to Proposition \ref{Propo8}, Equation (\ref{Propo8b}), the
parametric values of the superstable cycles of prime period $p\geq 2$ are
the branching points of certain on-axis parabolas of the form $\phi _{\pm
,-,\alpha }(t)$ with $\left\vert \alpha \right\vert =p-2$. These parabolas
originate precisely from the transversal intersections of the bisector with
the root branches $\phi _{+,\alpha }(t)$ (if any). As in Sections \ref{section3} and \ref{section4}, those branching points are denoted by $t_{\pm
,-,\alpha }$ ($=t_{+,-,\alpha }=t_{-,-,\alpha }$). Therefore, the parameters
of the superstable cycles can be ordered using the general ordering of the
root branches, Proposition \ref{Propo12}; alternatively, $t_{\pm ,-,\alpha
}<t_{\pm ,-,\beta }$ if and only if $\phi _{+,-,\alpha }(2)>\phi _{+,-,\beta
}(2)$. See Figure \ref{figure9} for the on-axis parabolas of ranks 2-5; in
case of equal branching points (e.g., $\phi _{\pm -}(t)$ and $\phi _{\pm
---}(t)$), only the lowest rank is shown because it corresponds to the prime
period. The branching points are ordered as follows: 
\begin{equation}
t_{\pm -}<t_{\pm -+-}<t_{\pm -+--}<t_{\pm -+}<t_{\pm -++-}<t_{\pm
-++}<t_{\pm -+++}  \label{t*}
\end{equation}corresponding, respectively, to the superstable cycles 
\begin{equation*}
+C,\;\;+-+C,\;\;+-\{+\}^{2}C,\;\;+-C,\;+\{-\}^{2}+C,\;\;+\{-\}^{2}C,\;\;+\{-\}^{3}C
\end{equation*}listed in Table \ref{table2} for prime periods 2-5. As shown in the proof of
Proposition \ref{Propo10}, the points $t_{\pm ,\alpha }$ (ordered as in
Proposition \ref{Propo12}) converge to $2$ as $\left\vert \alpha \right\vert \rightarrow \infty $.

\begin{figure}[tbp]
\centering
\includegraphics[height=4cm]{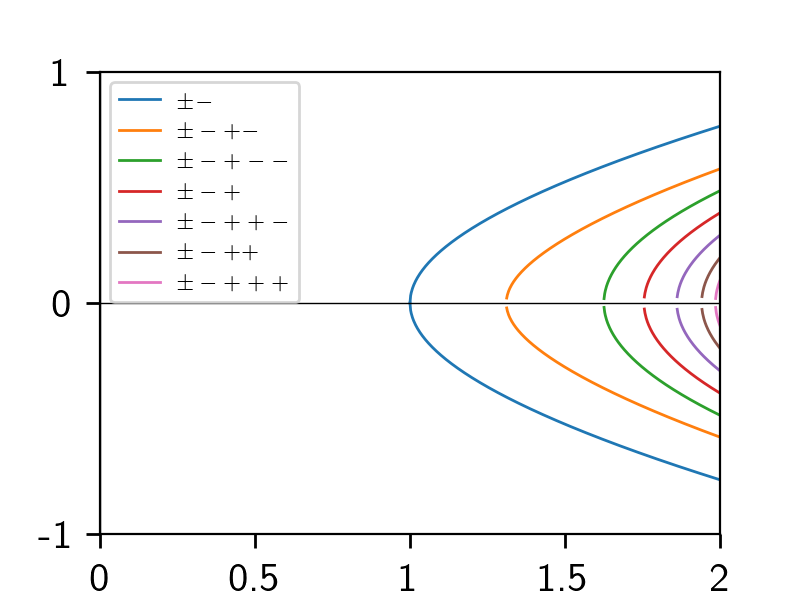}
\caption{On-axis root parabolas of ranks 2-5. In case of \ parabolas with
coinciding branching points, only the parabola with the lowest rank
(corresponding to the prime period of the superstable cycle) is shown. The
branching points (see Equation (\protect\ref{t*})) are: $t_{\pm -}=1$, $t_{\pm -+-}\simeq 1.3107$, $t_{\pm -+--}\simeq 1.6254$, $t_{\pm -+}\simeq
1.7549$, $t_{\pm -++-}\simeq 1.8607$, $t_{\pm -++}\simeq 1.9408$, and $t_{\pm -+++}\simeq 1.9854$. The first four parameter values (periods 2, 4, 5
and 3) are clearly visible in the bifurcation diagram, Figure 3, at the
intersection of periodic attractors with the axis $x=0$.}
\label{figure9}
\end{figure}

The superstable cycles of the quadratic family (and other three parametric
families of transformations of the interval) were numerically studied in 
\cite{Metropolis1973}. According to this paper, the number of superstable
cycles of the quadratic family is given in the following table.

\begin{table}[H]
	\caption{Number of superstable cycles of the quadratic family.}
	\label{table3}
	\centering
	\begin{tabular}{lcccccccccccccc}
	%\begin{tabular}{|l||c|c|c|c|c|c|c|c|c|c|c|c|c|c|}
	\toprule
	\textbf{Period} & 2 & 3 & 4 & 5 & 6 & 7 & 8 & 9 & 10 & 11 & 12 & 13 & 14 & 15 \\ \hline
	\textbf{\# sup. cycles} & 1 & 1 & 2 & 3 & 5 & 9 & 16 & 28 & 51 & 93 & 170 & 
	315 & 585 & 1091 \\ 
	\bottomrule
	\end{tabular} 
\end{table}

The parameter $t$ of a superstable cycle $(+,\sigma _{2},...,\sigma _{n},C)$
can be numerically computed by means of the computational loop\begin{equation*}
t_{\nu +1}=\phi _{+,-\sigma _{2},...,-\sigma _{n}}(t_{\nu }),
\end{equation*}$\nu =0,1,...$ until (i) $\left\vert t_{\nu +1}-\phi _{+,-\sigma
_{2},...,-\sigma _{n}}(t_{\nu })\right\vert <\varepsilon $, where $\varepsilon >0$ is the desired precision, or (ii) a prefixed maximum number
of loops $\nu _{\max }$ is reached, flagging that the convergence $t_{\nu
}\rightarrow t_{+,-\sigma _{2},...,-\sigma _{n}}$ is too slow.

\begin{Theorem}
\label{Theo11} (a) A symbolic cycle $\Sigma =(+,\sigma _{2},...,\sigma
_{s-1},C)$ can be admissible only for one value of $t$. (b) If $\Sigma $ is
admissible for $t=t_{1}$ and $\Sigma ^{\prime }=(+,\rho _{2},...,\rho
_{r-1},C)\neq \Sigma $ is admissible for $t=t_{2}$, then $t_{1}\neq t_{2}$.
\end{Theorem}

\begin{proof}
(a) Suppose $(+,\sigma _{2},...,\sigma _{s-1},C)$ is admissible for two
different parametric values $t_{1}$ and $t_{2}$. By Proposition \ref{Propo8}, $t_{1}$ and $t_{2}$ are then two branching points of $\phi _{+,-,-\sigma
_{2},...,-\sigma _{s-1}}(t)$, which contradicts Proposition \ref{Propo5B}.

(b) Suppose by contradiction that $t_{1}=t_{2}=t_{0}$. By Proposition \ref{Propo8}, 
\begin{equation*}
t_{0}=\phi _{+,-\sigma _{2},...,-\sigma _{s-1}}(t_{0})=\phi _{+,-\rho
_{2},...,-\rho _{r-1}}(t_{0}),
\end{equation*}where $t_{0}$ is a regular point. On the other hand, according to
Proposition \ref{Propo12}, root functions can coincide only at singular
points.
\end{proof}

\begin{Remark}
\label{RemarkFinal}The main ingredient in the proof of Theorem \ref{Theo11}
is the fact that $\mathbf{sdom\,}\phi _{\sigma }$ is a half-interval $(t_{\sigma },2]$ (Proposition \ref{Propo5B}), from which Milnor's
Monotonicity Conjecture (Theorem \ref{Theo9}) was derived. Reciprocally,
from Theorem \ref{Theo11} it follows that the bisector can transversally
intersect a positive root branch $\phi _{+,\sigma }(t)$ only once. In turn,
it recursively follows from this that the simplicity domains of the root
branches are half-intervals and, hence, Milnor's Monotonicity Conjecture.
\end{Remark}

Two maps of the interval $f(x)$ and $g(y)$ are called \textit{combinatorially equivalent} if they are conjugate via an order-preserving
transformation $\varphi (x)$. For instance, $q_{t}(x)$ and $cq_{1/c}(y)=1-cy^{2}$ are combinatorially equivalent via $\varphi (x)=\frac{1}{t}x$ and $c=t$, whereas $q_{t}(x)$ and $-q_{-c}(y)=c+y^{2}$ are conjugate
via $\varphi (x)=-x$ and $c=-t$, but they are not combinatorially equivalent
because $\varphi (x)$ reverses the order in this case. It is plain that
combinatorially equivalent maps have the same symbolic sequences for
corresponding initial conditions $x_{0}$ and $\varphi (x_{0})$.

\begin{Theorem}
\emph{(Thurston's Rigidity \cite{Bruin2015})} \label{TheoRigidity}Consider $q_{t_{1}}$ and $q_{t_{2}}$ for which their critical points $c=0$ have finite
orbits $\mathcal{O}$ and $\mathcal{O}^{\prime }$. If $q_{t_{1}}$ and $q_{t_{2}}$ are combinatorially equivalent, then $t_{1}=t_{2}$.
\end{Theorem}

\begin{proof}
Suppose that $q_{t_{1}}$ and $q_{t_{2}}$ are combinatorially equivalent via
an order-preserving conjugacy $\varphi $. Then, the symbolic sequence $\Sigma $ of $\mathcal{O}$ and the symbolic sequence $\Sigma ^{\prime }$ of $\varphi (\mathcal{O})=\mathcal{O}^{\prime }$ are equal$.$ Apply now Theorem \ref{Theo11}(a) to conclude $t_{1}=t_{2}$.
\end{proof}

As mentioned in the Introduction, Thurston's Rigidity implies Milnor's
Monotonicity Conjecture for the quadratic maps \cite{Bruin2015}. In Remark \ref{RemarkFinal} we sketched how this derivation could be done using
Theorem \ref{Theo11}, which is a sort of symbolic version of Thurston's
Rigidity.

\subsection{Dark lines and the Misiurewicz points}\label{section52}

To wrap up our excursion into the superstable cycles of the quadratic
family, let us remind that the \textquotedblleft dark
lines\textquotedblright\ in the bifurcation diagram (Figure \ref{figure3})
that go through\ the chaotic regions or build their boundaries are
determined by the superstable periodic orbits. To briefly study those dark
lines, we resort again to the polynomials $P_{n}(t)\equiv q_{t}^{n}(0)$
introduced in Equations (\ref{psi(t)}) and (\ref{phi(t)}).

We have already discussed in Section \ref{section51} how to pinpoint superstable cycles $(0,P_{1}(t),...,P_{p-1}(t))^{\infty }$ in the parametric interval using
symbolic sequences and root branches. More generally, consider orbits of $0$
that are eventually periodic, that is:\begin{equation}
(P_{n}(t))_{n=0}^{\infty
}=(0,t,P_{2}(t),...,P_{h-1}(t),(P_{h}(t),P_{h+1}(t),...,P_{h+T-1}(t))^{\infty }).  \label{dark}
\end{equation}Such parametric values are called \textit{Misiurewicz points} \cite{Misiu1991} and denoted as $M_{h,T}$, where we assume that $h\geq 1$ is the
minimal length of the preperiodic \textquotedblright tail\textquotedblright\
(the preperiod)\ and $T\geq 1$ is the prime period of the periodic cycle.
Therefore, if $M_{h,T}$ is a Misiurewicz point, then\begin{equation}
P_{h}(M_{h,T})=P_{h+T}(M_{h,T})=P_{h+2T}(M_{h,T})=...,  \label{phi_h}
\end{equation}so that the curves $P_{h+kT}(t)$, $k\geq 0$, meet at $t=M_{h,T}$ in the $(t,x)$-plane.

For example,\begin{equation}
(P_{n}(2))_{n=0}^{\infty }=(0,2,(-2)^{\infty }),  \label{psi(2)}
\end{equation}i.e., the orbit of $0$ hits a (repelling) fixed point after only two
iterations. Comparison of Equations (\ref{psi(2)}) and (\ref{dark}) shows
that $2=M_{2,1}$, therefore, all curves $P_{n}(t)$ with $n\geq 2$ meet at $t=2$ (see Equation (\ref{phi(t)B})).

\begin{figure}[tbp]
\centering
\includegraphics[height=4cm]{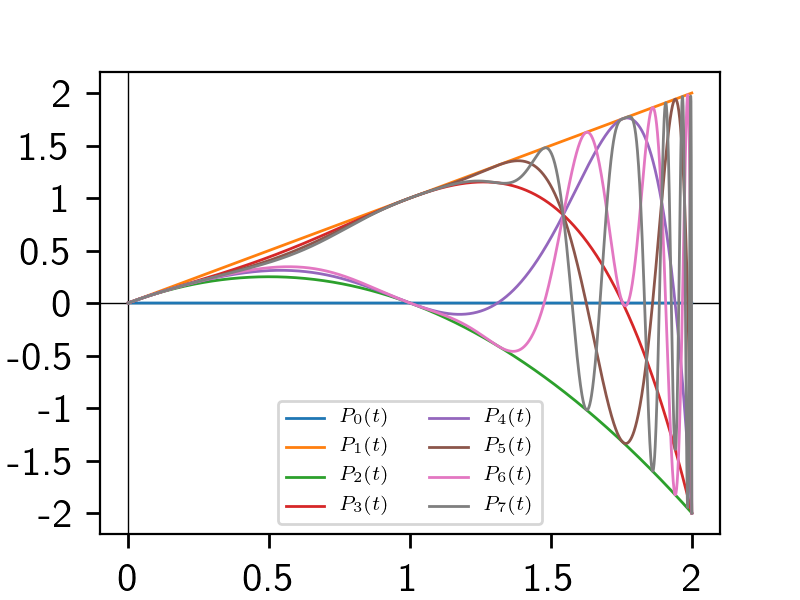}
\caption{The polynomials $P_{n}(t):=q_{t}^{n}(0)$ for $0\leq n\leq 7$.}
\label{figure10}
\end{figure}

The graphs of $P_{0}(t),...,P_{7}(t)$ are shown in Figure \ref{figure10}. As
a first observation, one can recognize the main features of the chaotic band$\allowbreak $s in the bifurcation diagram of the quadratic family, in
particular band merging. We also see that the curves $x=P_{n}(t)$ intersect
transversally or tangentially; all these intersections are related to
important aspects of the dynamic. Chaos bands merge where those curves
intersect transversally, while periodic windows open where they intersect
tangentially the upper and lower edges. Moreover, the functions $P_{n}(t)$
intersect the $t$-axis precisely at the parameter values for which $0$ is
periodic:\begin{equation*}
P_{n}(t_{0})\equiv q_{t_{0}}^{n}(0)=0\;\;\Leftrightarrow \;\;\phi _{\sigma
_{1},...,\sigma _{n}}(t_{0})=0.
\end{equation*}Besides $P_{n}(0)=0$ for all $n\geq 0$ (Equation (\ref{psi(0)})) and $P_{n}(1)=0$ for all $n\geq 1$ (Equation (\ref{psi(t)})), the following zeros
of $P_{n}(t)$ can be read in Figure 9: $P_{3}(t)=0$ at $t\simeq 1.7549$; $P_{4}(t)=0$ at $t=1.3107$ and $1.9408$; and $P_{5}(t)=0$ at $t=1.6254$, $1.8607$ and $1.9854$.

As way of illustration, we will calculate $M_{3,1}$, the perhaps most
prominent Misiurewicz point in Figure \ref{figure10}, which corresponds to
the merging of the two chaotic bands into a single band. By Equation (\ref{phi_h}) with $h=3$ and $T=1$,  
\begin{equation*}
P_{3}(M_{3,1})=P_{4}(M_{3,1})=P_{5}(M_{3,1})=...
\end{equation*}From $P_{3}(M_{3,1})=P_{4}(M_{3,1})$ and Equation (\ref{phi(t)}) if follows
that $M_{3,1}$ is the unique real solution in $(0,2)$ of the equation 
\begin{equation*}
4-6t+6t^{2}-4t^{3}+t^{4}=0,
\end{equation*}namely, $M_{3,1}=1.5436890..$. For more in-depth information, the interested
reader is referred to \cite{Romera1996,Hutz2015}.

Among the many remarkable properties of the Misiurewicz points, we highlight
only the following two: (i) the periodic cycle $(P_{h}(M_{h,T}),P_{h+1}(M_{h,T}),...,P_{h+T-1}(M_{h,T}))^{\infty }$ in
Equation (\ref{dark}) is repelling \cite{Douady1984}, and (ii) $q_{t}(x)$
has an absolutely continuous invariant measure for each $t=M_{h,T}$ \cite{Misiu1981,Rychlik1992}.

\section{Conclusion}\label{section6}

In the previous sections we have revisited two classical topics of the
continuous dynamics of interval maps: entropy monotonicity (Section \ref{section4}) and
superstable cycles (Section \ref{section5}) for the quadratic family $q_{t}(x)$ (Section \ref{section3}). The novelty consists in the starting point: we use Equation (\ref{h(q_t)}) for $h(q_{t})$, the topological entropy of $q_{t}$, where $s_{n}(t)$ is
the number of transversal intersections of the polynomial curves $q_{t}^{n}(x)$ with the $x$-axis. Equation (\ref{h(q_t)}) and several
numerical schemes for its computation were derived in \cite{Amigo2012,Amigo2014,Amigo2015}. This approach leads directly to the root
functions ($\phi _{\sigma }(t)$), bifurcation points ($t_{\sigma }$) and
smoothness domains ($\mathbf{sdom\,}\phi _{\sigma }$) studied in Sections \ref{section31} and \ref{section32}. It is precisely the structure of the smoothness domains, $\mathbf{sdom\,}\phi _{\sigma }=(t_{\sigma },2]$ (Proposition \ref{Propo5B}), which
implies that $s_{n}(t)$ is a nondecreasing staircase function for each $n\geq
1$ (Theorem \ref{Theo8}) and, in turn, that the function $t\mapsto h(q_{t})$ is monotone
(Theorem \ref{Theo9}). Unlike existing proofs \cite{Douady1984,Milnor1988,Douady1995,Tsujii2000}, Theorem \ref{Theo9} proves Milnor's
Monotonicity Conjecture via real analysis. This also shows that the
transversal intersections of a multimodal map and its iterates with the
critical lines is a useful tool in one-dimensional dynamics. Sections \ref{section22} and \ref{section4} contains further details on Milnor's Monotonicity Conjecture and its
generalization to multimodal maps.

In Section \ref{section51} we derived some basic results on the superstable cycles of the
quadratic family, in particular Theorem \ref{Theo11}, which is a sort of
symbolic version of Thurston Rigidity (Theorem \ref{TheoRigidity}). The
commonalities between entropy monotonicity and the superstable cycles of the
quadratic maps go beyond the techniques used, namely, root branches,
bifurcation points, transversality, and a geometrical language. There is
also a flow of ideas in both directions. We started with the topological
entropy and worked our way towards the superstable cycles, but the other
direction works too, although we only indicated this possibility in Remark \ref{RemarkFinal}. We also made a brief excursion into the preperiodic
orbits of the critical point in Section \ref{section52} (Misiurewicz points). In
conclusion, both topics complement and intertwine in remarkable ways, as
well as being interesting on their own.

%%%%%%%%%%%%%%%%%%%%%%%%%%%%%%%%%%%%%%%%%%%%%%%
%%%%%%%%%%%%%%%%%%%%%%%%%%%%%%%%%%%%%%%%%%
\section*{Acknowledgments}
This work was financially supported by the Spanish Ministry of Science and Innovation, grant PID2019-108654GB-I00.

\bibliographystyle{unsrt}  
\bibliography{references}  
%%% Remove comment to use the external .bib file (using bibtex).

%%%%%%%%%%%%%%%%%%%%%%%%%%%%%%%%%%%%%%%%%%%%%%%
\end{document}